\documentclass[12pt]{amsart}
\usepackage{graphicx} 
\usepackage{amsmath,amssymb, amsfonts, amscd}
\usepackage{tikz-cd, tikz}
\usepackage{subcaption}
\usepackage{amsthm, bbm}
\usepackage{appendix}
\usepackage[normalem]{ulem}

\usepackage[a4paper, margin=3cm]{geometry}
\usepackage{xfrac}

\usepackage{xcolor}
\usepackage{framed}

\usepackage{hyperref}
\usepackage[capitalise]{cleveref}

\theoremstyle{plain}
\newtheorem*{theorem*}{Theorem}
\newtheorem*{remark*}{Remark}
\newtheorem*{example*}{Example}
\newtheorem*{conjecture*}{Conjecture}

\newtheorem{lemma}{Lemma}[subsection]

\newtheorem{thm}[lemma]{Theorem}

\newtheorem{prop}[lemma]{Proposition}

\newtheorem{crl}[lemma]{Corollary}

\newtheorem{question}[lemma]{Question}

\newtheorem{introtheorem}{Theorem}

\newtheorem{introthm}[introtheorem]{Theorem}

\newcommand{\restate}[3]{
    \noindent \textbf{#1 \textup{#2}.} #3
}

\theoremstyle{definition}
\newtheorem{definition}[lemma]{Definition}
\newtheorem*{definition*}{Definition}

\newtheorem{example}[lemma]{Example}

\theoremstyle{remark}
\newtheorem{remark}[lemma]{Remark}

\newcommand{\Hom}{\operatorname{Hom}}
\newcommand{\pic}{\operatorname{Pic}}
\newcommand{\fun}{\operatorname{Fun}}
\newcommand{\eps}{\epsilon}

\newcommand{\B}{\mathcal{B}}
\newcommand{\C}{\mathcal{C}}
\newcommand{\1}{\mathbbm{1}}
\newcommand{\D}{\mathcal{D}}
\newcommand{\T}{\mathcal{T}}
\newcommand{\Z}{\mathcal{Z}}
\newcommand{\E}{\mathcal{E}}
\newcommand{\M}{\mathcal{M}}
\newcommand{\N}{\mathbb{N}}
\newcommand{\End}{\text{End}}
\newcommand{\FPdim}{\operatorname{FPdim}}
\newcommand{\id}{\operatorname{id}}
\newcommand{\Rep}{\operatorname{Rep}}
\newcommand{\Vect}{\operatorname{Vec}}
\newcommand{\kk}{\mathbbm{k}}
\newcommand{\sVec}{\operatorname{sVec}}
\newcommand{\comment}[1]{}

\newcommand{\sRep}{\operatorname{sRep}}
\newcommand{\Mod}{\operatorname{Mod}}
\newcommand{\triv}{\operatorname{triv}}

\tikzset{commutative diagrams/every label/.append style={yshift=3pt}} 

\title[braided non-semisimple near-group categories]{On braided simple extensions and braided non-semisimple near-group categories}
\author{Daniel Sebbag}
\address{Department of Mathematics, Ben-Gurion University of the Negev, Beer-Sheva, Israel}
\email{sebbagd@post.bgu.ac.il}
\date{\today}

\begin{document}

\begin{abstract}
We study simple extensions of pointed finite tensor categories, that is, tensor categories $\C$ admitting an abelian decomposition $\C \cong \D \oplus \mathcal{M}$ where $\D$ is a pointed tensor subcategory and $\mathcal{M}$ is semisimple with a unique simple projective object. Such categories provide a natural generalization of near-group categories. Our results concern the braided case.

We prove that every non-degenerate braided non-semisimple near-group category is a braided simple extension of $\sRep(W\oplus W^*)$ with non-trivial braiding for which $\sRep(W)$ is Lagrangian, and any braided non-semisimple near-group category $\C$ arises canonically as an extension of a non-degenerate braided tensor category with a Lagrangian subcategory $\sRep(W)$ and $\Rep(K)$, where $K$ is isomorphic to the Picard group of a symmetric subcategory determined by the unique simple 
projective object of $\C$.
\end{abstract}

\maketitle
\noindent\textbf{Keywords:} finite tensor categories, near-group categories, non-semisimple categories, braided tensor categories

\noindent\textbf{MSC 2020:} 18M05, 18M15.

\section{Introduction}
Let $\D$ be a finite tensor category; that is, an abelian, $\kk$-linear, rigid monoidal category with a finite number of (isomorphism classes of) simple objects, where $\1$ is a simple object and each object has a projective cover. For a monoidal category $\C$, a category $\M$ is called a left $\C$-module category if there exists a bifunctor $\tilde\otimes: \C\times \M \to \M$ and a natural isomorphism 
$$
    m_{X,Y,M}:(X\otimes Y)\tilde\otimes M\to X\tilde\otimes (Y\tilde\otimes M)
$$
for $X,Y\in \C$ and $M\in \M$ which is compatible with the associativity constraint of $\C$ and $M\mapsto \1\tilde\otimes M$ is an autoequivalence of $\M$. 

A classical result in the theory of finite tensor categories is that if $\D$ has a fiber functor, then $\Vect$ is a semisimple $\D$-bimodule category, and $\C:=\D\oplus \Vect$ is also a $\D$-bimodule category (see \cite{EGNO}). A natural question that arises is: when can we define a monoidal structure on $\C$. To be more precise, we ask: 
\begin{question}
    For a finite tensor category $\D$ with a fiber functor $F:\D\to \Vect$, is there a category $\M$ which is equivalent to $\Vect$ as a $\D$-bimodule category, such that $\C:=\D\oplus \M$ has a structure of a finite tensor category (which is compatible with the $\D$-action on $\M$)?
\end{question}
When the answer to this question is yes, we say that $(\C,\D,\M)$ is a simple extension of $\D$. Usually, we will omit $\M$ and simply say that \textit{$\C$ is a simple extension of $\D$}. Notice that if $\C$ is a simple extension of $\D$, then $\C$ has a simple projective object $Q$, the unique simple object in $\M$. When $\C$ and $\D$ are braided, such that the braiding of $\C$ coincide with that of $\D$, we call $(\C,\D,\M)$ a \textit{braided} simple extension.

In fusion categories, there is a family of categories called \textit{near-group categories}, which were introduced in \cite{S1} (though classification of such categories appeared even earlier in \cite{TY}). A near-group category is a fusion category $\C$ with a unique non-invertible object $Q$. Denoting the group of simple invertible objects of $\C$ by $G = \{L_i\}$, then a key parameter in the structure of $\C$ is the parameter $r\in \N$ such that $Q\otimes Q\cong \oplus_{i}L_i \oplus rQ$. This near-group category is usually denoted by $\C(G,r)$, and we say that $\C$ has a fusion rule $(G,r)$.

There are many known results regarding near-group categories. For instance, \cite{TY} gives a classification of near-group category with fusion rule $(G,0)$. In addition, \cite{S2} classified braided near-group categories with fusion rule $(G,0)$, and their braid structures (including the structure of the underlying group $G$), and \cite{T} classified braided near-group categories with fusion rule $(G,r)$ for $r>0$. These results strongly rely on two factors:
\begin{itemize}
    \item The semisimplicity of near-group categories, which enables translating the diagrams defining (braided) fusion categories to equations describing the fusion rules $(G,r)$, due to their unique structure (as a fusion category with a unique non-invertible object).
    \item The classification of pointed fusion categories, as categories of the form $\Vect^\omega_G$ of $G$-graded vector spaces, with associativity constraint corresponding to a $3$-cocycle $\omega\in H^3(G,\kk^*)$.
\end{itemize}
A natural way to generalize near-group categories to the class of finite tensor categories is:
\begin{definition}
    A non-semisimple finite tensor category $\C$ is called a \textit{non-semisimple near-group category} if $\C$ is a simple extension of a non-semisimple pointed finite tensor category $\D$.
\end{definition}
Extending the notion from fusion categories, we say that a non-semisimple near-group category $\C$ has \textit{generalized fusion rule} $(G,r)$ if $G$ is the group of invertible objects of $\C$ (that is, the Picard group of $\D$, of which $\C$ is its simple extension) and $r$ is the number of copies of the unique simple projective object $Q$ in the tensor product $Q\otimes Q$.

A full classification of non-semisimple near-group categories is out of reach, and such a classification might require first a classification of pointed non-semisimple finite tensor categories, which doesn't exist yet. Another factor that makes this classification harder is the fact that in a non-semisimple near-group category, the tensor product of the unique non-invertible simple object $Q$ is not a direct sum of simple objects, but a direct sum of projective covers of simple objects, as we show in \cref{eq: Q times Q} in \cref{sec: braided simple projective extensions}. In particular, the diagrams defining a finite tensor category don't translate to equations of scalar parameters as in fusion near-group categories. A basic example of such a category is the category $\Rep(A''_{C_4})$ for $A''_{C_4}$ the Hopf algebra defined in \cite{M}\footnote{$A''_{C_4}$ is the Hopf algebra generated by a grouplike element $g$ of order $4$, a $(g,1)$-skew-primitive element $x$ such that $x^2=g^2-1$ and $xg=-gx$}, which is a simple extension of $\Rep(H_4)$ for $H_4$ the Sweedler's Hopf algebra.

In this paper, we wish to take a first step toward classifying non-semisimple near-group categories, by focusing on \textit{braided} non-semisimple near-group categories, i.e. \textit{braided simple extensions} of braided pointed finite tensor categories, admitting a fiber functor. The classification of braided pointed finite tensor categories with a fiber functor (see \cite{BN1}) allows us to derive structural results about braided non-semisimple near-group categories (see \cref{sec: simple extensions of pointed FTC}).

Formally, in this work we consider triplets of simple extensions $(\C,\D,\M)$ for which $\D$ is braided, $\M$ is equivalent as a $\D$-bimodule category to $\Vect$ and $\C$ is the simple extension of $\D$, such that $\C$ is braided and $\D$ is a braided tensor subcategory of $\C$. A concrete example for such a category is given in \cite[Appendix A]{GS}. On the other hand, there are braided finite tensor categories with simple extensions that are not braided - for instance  $\Rep(H_4)$ is braided but $\Rep(A''_{C_4})$ is not, because $A''_{C_4}$ has no quasi-triangular structure (see \cite{M}). 

In the fusion case, \cite{T} proves that there are braided near-group categories with fusion rule $(G,r)$ for $r>0$. In this paper, we prove that this fails in non-semisimple near-group categories. That is, we prove the following theorem:
\begin{introthm}\label{thm: non semisimple braided near-group categories have fusion rule 0}
    Let $\C$ be a non-semisimple braided near-group category with generalized fusion rule $(G,r)$, then $r=0$, and consequently $\C$  is weakly integral.
\end{introthm}
One of the steps in proving this theorem is showing that in a braided simple extension $(\C,\D,\M)$, $\D$ is never contained in the Müger center of $\C$ (see Section 3.3). This is significantly different from the fusion case, where $\Rep(S_3)$ is a braided simple extension of its central subcategory $\Rep(\mathbb{Z}/2)$.

In their work (\cite{BN2}, \cite{BN3}), the authors defined the notion of a central exact sequence of finite tensor categories as an exact sequence $\E\xrightarrow{i} \C\xrightarrow{F} \D$ of finite tensor categories with exact tensor functors $i:\E\to \C$ and $F:\C\to \D$, which give a categorical analogue to exact sequences (for more details, see \cref{sec: braided exact sequences}). Such an exact sequence is called braided if the categories $\E,\C,\D$ and the functors $i, F$ are braided, and it is called a \textit{modularization exact sequence} if $\D$ is non-degenerate. 

In their work, \cite{GS} classified non-degenerate braided finite tensor categories with a Lagrangian subcategory braid equivalent to $\sRep(W)$ (the category of supermodules of the exterior algebra over a purely odd supervector space $W$), and they proved that some of these categories are simple braided extensions of $\sRep(W\oplus W^*)$ (though they do not use the term "braided simple extension"). We call such categories \textit{Purely odd of Lagrangian type}, or simply abbreviate them as POL-categories.

In this paper, we prove that a braided non-semisimple near-group category is an "extension" of a purely odd of Lagrangian type category by a symmetric fusion category. The essence of \cref{thm: exact sequence of braided near-group categories} is:

\begin{introthm}\label{thm: exact sequence of braided near-group categories}
    Let $\C$ be a braided non-semisimple near-group category, then $\C$ fits in a modularization exact sequence:
    \begin{align*}
        \Rep(K)\xrightarrow{i}\C\xrightarrow{F}\tilde{\C}.
    \end{align*}
    for $K$ an abelian finite group such that $\Rep(K)\cong \Z_2(\C)$ and $\tilde{\C}$ is a POL-category.
\end{introthm}
The proof of this theorem is constructive, and $(\tilde{\C},K)$ are determined by $\C$. 

An immediate result of this Theorem is that every \textit{non-degenerate} braided non-semisimple near-group category is purely odd of Lagrangian type, as shown in \cref{sec: simple extensions of pointed FTC}. Using \cref{thm: exact sequence of braided near-group categories}, we also prove that for a braided non-semisimple near-group category $\C$ with generalized fusion rule $(G,0)$ (the only option by \cref{thm: non semisimple braided near-group categories have fusion rule 0}), the group $G$ is an abelian $2$-group, which is compatible with the fusion case \cite{S2}.
\subsection{Structure of the paper}
In \cref{sec: Preliminaries}, we give the necessary preliminaries, including results concerning the M\"uger centralizer in braided finite tensor categories, which will be used extensively.

In \cref{sec: braided simple projective extensions} we consider general braided simple extensions $(\C,\D,\M)$, and develop some basic tools that will be required to the study of braided non-semisimple near-group categories. In \cref{sec: simple extensions of pointed FTC} we focus on braided non-semisimple near-group categories.

\subsection*{Acknowledgments} I'm grateful to Inna Entova-Aizenbud for helpful discussions, and to Pavel Etingof for his comments on the paper. This research was supported by the Israel Science Foundation grant ISF No. 1362/23 (PI: Inna Entova-Aizenbud).

\subsection*{Declaration of AI use} During the preparation of this work the author used Claude Opus 5.5 (Anthropic). The tool was used to point out errors in an earlier version of this paper, to suggest proof strategies, in particular in \cref{sec: no extensions to symmetric categories}. All mathematical statements and proofs were verified by the author, who reviewed and edited all AI-assisted content and takes full responsibility for the originality, correctness and integrity of the paper, including its references.

\section{Preliminaries}\label{sec: Preliminaries}
In this section, we recall some basic background on finite tensor categories and braiding of tensor categories. Throughout this paper, we work over $\kk$, an algebraically closed field of characteristic zero, and all categories are assumed to be abelian and $\kk$-linear. 

    \subsection{Notation and terminology} 
    By a finite tensor category $\C$, we mean a $\kk$-linear, abelian, monoidal category with finitely many simple objects (up to isomorphism), such that each simple object has a projective cover, and $\1$ is a simple object. Throughout this paper, we assume all the tensor categories have trivial unit constraints, but not necessarily trivial associativity constraint. $\C$ is called \textit{semisimple} if every object is a direct sum of simple objects.

    By a tensor subcategory $\D\subset \C$, we mean a full subcategory which is closed under taking subquotients, tensor products and duals, as in \cite[4.11.1]{EGNO}. This coincide with the definition of \textit{tensor full} subcategory in \cite{Sh}. 
    
    In a finite tensor category, every object $X$ has a socle filtration $0=X_0\subset X_1 \subset ... \subset X_n=X$ in which $X_{i+1}/X_i$ is the sum of simple subobjects of $X/X_i$. For $X\in \C$, we define the \textit{Loewy length} to be the length of the socle filtration of $X$.  

    A tensor functor $F:\C\to \D$ is called \textit{surjective} if every $Y\in \D$ is a subquotient of $F(X)$ for some $X\in \C$. Given a finite tensor category $\C$, a \textit{fiber functor} is an exact faithful functor $F:\C\to \Vect$ such that $F(\1) = \kk$ with a family of isomorphisms $\{J_{X,Y}:F(X)\otimes F(Y)\to F(X\otimes Y)|X,Y\in \C\}$ which satisfy the axioms of a monoidal functor.

    By definition, for a simple extension $(\C,\D,\M)$, $\D$ has a bimodule category $\M$ which is semisimple with a unique simple object $Q$, so it has a fiber functor $F:\D\to \Vect$ given by $\omega_Q(X) = \Hom_\M(Q,X\otimes Q)$. This implies that $\D$ is equivalent (as a monoidal category) to a category of representations of some Hopf algebra, thus has trivial associativity constraint.

    \subsection{Braiding, symmetry, twists and M\"uger centralizer}
        A tensor category $\C$ is called \textit{braided} if there exist isomorphisms $\{c_{X,Y}:X\otimes Y \to Y\otimes X\}_{X,Y\in \C}$ such that the following holds for objects $X,Y,Z\in \C$ and morphisms $\psi:X\to Z$:
        \begin{equation}\label{eq: composition of braiding}
            c_{X\otimes Y, Z} = a_{Z,X,Y}\circ(c_{X,Z}\otimes \id_Y)\circ a^{-1}_{X,Z,Y}\circ(\id_X\otimes c_{Y,Z})\circ a_{X,Y,Z} ,
        \end{equation}
        \begin{equation}\label{eq: naturallity of braiding}
            (\id_Y \otimes \psi)\circ c_{X,Y} = c_{Z,Y}\circ (\psi\otimes \id_Y)
        \end{equation}
        for $a$ the associativity constraint of $\C$. \cref{eq: composition of braiding}, \cref{eq: naturallity of braiding} have analogues for $c_{X,Y\otimes Z}$ and naturality in the second variable as well. For more details, see for instance \cite{EGNO}.

    For a braided finite tensor category $\C$, let $s$ be the squared braiding, i.e., $s_{X,Y}:=c_{Y,X}\circ c_{X,Y}$ for every $X,Y\in\C$. Recall that two objects $X,Y\in{\C}$ {\em centralize} each other if $s_{X,Y}= \id_{X\otimes{Y}}$, and that the {\em centralizer} $\D_\C'$ of $\D\subseteq\C$ is the full subcategory of $\C$ consisting of all objects which centralize every object of $\D$ (see, e.g., \cite{DGNO}). 
    \begin{definition}\label{def: muger center}
        For a braided finite tensor category $\C$, the centralizer $\C'$ of $\C$ is called \textit{M\"uger center} of $\C$ and denoted by $\Z_2(\C)$.
    \end{definition}
    The centralizer enables us to distinguish between different types of braided finite tensor categories. For finite braided tensor categories $\D\subset\C$ we say that:
    \begin{itemize}
        \item $\D$ is symmetric if $D\subseteq \Z_2(\D)$. 
        \item $\D$ is Lagrangian if $\D=\D_\C'$.
        \item $\D$ is central in $\C$ if $\D_\C'= \C$ (thus a central category is always symmetric).
        \item $\C$ is non-degenerate if $\Z_2(\C)=\Vect$.
    \end{itemize}

    In general, for $X,Y\in \C$ we have an equality
    $$
        s_{X,Y} = c^{-1}_{X,Y}\circ s_{Y,X}\circ c_{X,Y}
    $$    
    so $s_{X,Y} = \id_{X\otimes Y} \Leftrightarrow s_{Y,X} = \id_{Y\otimes X}$, in which case $X,Y$ centralize each other.
    
    By \cite{Sh}, we have in general for $\D\subset \C$ a tensor subcategory
    \begin{equation}\label{eq: shimizu}
        \FPdim(\D)\FPdim(\D_\C')=\FPdim(\C)\FPdim(\Z_2(\C)\cap \D)
    \end{equation}   
    Notice that for three categories $\E\subset\D\subset\C$,  $\E'_\D$ might differ from $\E'_\C$ or $\Z_2(\E)$. In all cases, the centralizer in the largest category will always be denoted without the subscript, i.e., in this case, $\E'_\C$ will simply be denoted as $\E'$, similarly $\D'_\C$ will be denoted as $\D'$, and $\Z_2(\C)$ will be denoted $\C'$. Clearly, if $\E\subset \D\subset \C$ then $\Z_2(\E)\subset \E'_\D\subset \E'$.

    For any $X,Y,Z\in \C$ we have the following equality, which is a direct calculation using the definition of braiding:
    \begin{equation} \label{eq: squared braiding on (X times Y) times Z}
        s_{X\otimes Y, Z} = a^{-1}_{X,Y,Z}\circ (\id_X\otimes c_{Z,Y}) \circ a_{X,Z,Y}\circ (s_{X,Z}\otimes \id_Y) \circ a^{-1}_{X,Z,Y}\circ (\id_X\otimes c_{Y,Z})\circ a_{X,Y,Z} 
    \end{equation}
    \begin{equation} \label{eq: squared braiding on X times (Y times Z)}
        s_{X, Y\otimes Z} = a_{X,Y,Z}\circ (c_{Y,X}\otimes \id_Z) \circ a^{-1}_{Y,X,Z}\circ (\id_Y\otimes s_{X,Z}) \circ a_{Y,X,Z}\circ (c_{X,Y}\otimes \id_Z)\circ a^{-1}_{X,Y,Z}
    \end{equation}
    \begin{remark}
        Similarly, from \cref{eq: squared braiding on X times (Y times Z)} we get that if $s_{X,Y} = \lambda\cdot \id_{X\otimes Y}$ and $s_{X,Z} = \mu\cdot \id_{X\otimes Z}$, then $s_{X,Y\otimes Z} = \lambda\mu\cdot \id_{X\otimes Y\otimes Z}$.
    \end{remark}
    \begin{remark}
    \begin{enumerate}
        \item The centralizer of any family of objects is a tensor subcategory \cite{Sh}. Notice that it need not be closed under extensions.
        \item If $s_{X,Z} = \lambda\cdot\id_{X\otimes Z}$ and $s_{Y,Z} = \mu\cdot\id_{Y\otimes Z}$, then $s_{X\otimes Y,Z} = \lambda\mu\cdot\id_{X\otimes Y\otimes Z}$ and $s_{X*,Z} = \lambda^{-1}\cdot\id_{X^*\otimes Z}$.
    \end{enumerate}
    \end{remark}
    \subsection{De-equivariantization}\label{subsection: de-equivariantization}
    In this subsection, we briefly recall the notion of de-equivariantization and some results which we will use in the classification of braided non-semisimple near-group categories. For more details, see \cite{EGNO}. 
    
    For a finite tensor category $\C$, and an algebra object $A\in \C$, $\Mod_\C(A)$ is the category of \textit{right} $A$-modules in $\C$. For $X,Y\in \Mod_{\C}(A)$, we denote by $\Hom_A(X,Y)$ the vector space of $A$-module morphisms $f:X\to Y$ which are morphisms in the category $\C$. Notice that in general, we don't have $\Hom_A(X,Y) = \Hom_{\C}(X,Y)$. For $X\in \C$ and $Y\in \Mod_\C(A)$ we do have:
    $$
        \Hom_A(X\otimes A,Y) \cong \Hom_\C(X,U(Y)) 
    $$
    for $U(Y)$ the image of $Y$ under the forgetful functor $U:\Mod_\C(A)\to \C$ \cite[Lemma 7.8.12]{EGNO}. This adjunction is given by the formulas:
    $$
        \varphi\mapsto\varphi\circ (\id_X\otimes \eta),\quad \psi\mapsto \mu_Y\circ(\psi\otimes \id_A)
    $$
    for $\varphi\in \Hom_A(X\otimes A,Y),\psi\in \Hom_\C(X,U(Y))$ and $\mu_Y$ the $A$-module structure of $Y$. The category $\Mod_\C(A)$ is a left $\C$-module category, and when this module category is exact (in the sense of \cite{EO}) we call $A$ \textit{an exact algebra}.

    When the algebra object $A$ is commutative, $A\in \Z_2(\C)$, exact and $\Hom_\C(1,A) = \kk$, $\Mod_\C(A)$ has a structure of a finite tensor category with tensor product $\otimes_A$, and the free module functor $F:\C\to \Mod_\C(A)$ given by $F(X) = X\otimes A$ is a surjective tensor functor. We will use the following proposition which is proved in \cite{SY}
    
    \begin{prop}\label{prop: de-equivariantizaion of braided finite tensor categories}
    Let $\C$ be a braided finite tensor category with braiding $c$, and assume that $A\in \Z_2(\C)$ is an exact commutative algebra with $\Hom_\C(\1,A) = \kk$, then $\Mod_{\C}(A)$ is a braided finite tensor category with braiding $\tilde{c}:M\otimes_{A} N\to N\otimes_{A} M$ such that the diagram
    \begin{center}
        \begin{tikzcd}
            M\otimes N \arrow[r, "c_{M,N}"] \arrow[d, " "] & N\otimes M \arrow[d, " "] \\
            M\otimes_{A} N\arrow[r, "\tilde{c}_{M,N}"] & N\otimes_{A} M
        \end{tikzcd}
    \end{center}
    commutes, and the free module functor functor $F:\C\to \Mod_{\C}(A)$ is braided.
    \end{prop}

    For a braided finite tensor category $\C$ and a finite group $G$ such that $\Rep(G)\subset \Z_2(\C)$ is tannakian, $\fun(G,\kk)$ is commutative, exact and $\Hom_\C(\1,\fun(G,\kk))=\kk$, thus
    \cref{prop: de-equivariantizaion of braided finite tensor categories} holds. In this case we denote $\Mod_{\C}(\fun(G, \kk))$ as $\C_G$, and call this category \textit{the $G$-de-equivariantization of $\C$}.

    \begin{remark}
        As in the case of fusion categories, if $\C$ has a $G$-de-equivariantization, then we have $(\C_G)^G\simeq \C$, for $\D^G$ the $G$-equivariantization of $\D$ by $G$ (see \cite{EGNO, BN3}).  
    \end{remark}
    \begin{prop}\label{prop: invertible objects under de-equivariantization}
        Let $\C$ be a braided finite tensor category with de-equivariantization $\C_G$, $X\in \C$ a simple object. If $X\in \C$ is invertible, then $X\otimes A$ is invertible in $\C_G$.
    \end{prop}
    \begin{proof}
        The free module functor is a surjective tensor functor, so by \cite{EO} it preserves Frobenius-Perron dimensions, thus $\FPdim_{\C_G}(X\otimes A) = 1$, hence by \cite{EGNO}, $X\otimes A$ is invertible in $\C_G$.
    \end{proof}

    \subsection{Braided exact sequences of tensor categories}\label{sec: braided exact sequences}
    In this subsection, we recall the notion of braided exact sequences of finite tensor categories, which was introduced in \cite{BN3}. 

    A tensor functor $F:\C\to \D$ is called \textit{normal} if any object $X\in \C$ admits a subobject $X'\in \C$ such that $F(X')$ is the largest subobject of $F(X)$ isomorphic to $(\1_{\D})^{\oplus n}$ for $n\in \N$. We denote the full tensor subcategory of $\C$ generated by objects $X$ such that $F(X)\in \Vect$ by $\mathfrak{Ker}_F$.

    For $\C_1,\C_2,\C_3$ finite tensor categories over $\kk$, and tensor functors $i:\C_1\to \C_2$, $F:\C_2\to \C_3$, the sequence of tensor categories:
    $$
        \C_1\xrightarrow{i} \C_2 \xrightarrow{F} \C_3
    $$
    is called an \textit{exact sequence of tensor categories} if $F$ is surjective and normal, and $i$ is a full embedding which is a tensor equivalence between $\mathfrak{Ker}_F$ and $\C_1$.
    \begin{remark}
        Notice that this definition requires in particular that $\C_1$ admits a fiber functor. The notion of exact sequences was generalized to avoid this restriction by \cite{EG}, but in this work, we will consider exact sequences as defined in \cite{BN3}.
    \end{remark}

    An exact sequence of tensor categories is called \textit{braided} if the categories $\C_1,\C_2,\C_3$ are braided and the functors $i,F$ are braided. A braided exact sequence is called \textit{modularization exact sequence} if $\C_3$ is non-degenerate.
    \begin{remark}
        If $\E\xrightarrow{i} \C \xrightarrow{F} \D$ is a braided exact sequence, then $\E\subset \Z_2(\C)$, that is, $\E$ lies in the M\"uger center of $\C$. 
    \end{remark}
    \begin{example}\label{example: modularization sequence}
        Assume that $\C$ is a braided finite tensor category, $G$ a finite group and the conditions of \cref{prop: de-equivariantizaion of braided finite tensor categories} hold, then the sequence:
        $$
            \Rep(G)\xrightarrow{i}\C\xrightarrow{F}\C_G
        $$
        is a braided exact sequence, for $F:\C \to \C_G$ the free module functor $F(X) = X\otimes A$ for $A = \fun(G, \kk)$ the algebra object in $\Rep(G)$ \cite{BN3}.
    \end{example}

    \subsection{Symmetric tensor categories}
    A well known result proved in \cite{D} is the classification of symmetric finite tensor categories\footnote{This classification is for symmetric finite tensor categories of \textit{moderate growth}, but all our categories are such.} over an algebraically closed field of characteristic $0$. We give a short introduction and notation.

    A symmetric finite tensor category is a category of the form $\sRep(G,u)$ for $G = K\ltimes W$ a supergroup such that $K$ is a finite group, $W$ a finite dimensional purely odd super vector space and $u\in K$ a central element with $u^2 = 1$, acting on $W$ as the parity. The objects in this category are $G$-supermodules on which $u$ acts as the parity operator, and the morphisms are \textit{even morphisms} (i.e., morphisms that respect the $\mathbb{Z}_2$-grading of the objects). 

    Two important cases we will work with are: 
    \begin{itemize}
        \item $\dim(W) =0$ and $u=1$, in which case $\sRep(G,u)$ is simply $\Rep(K)$ with its usual symmetric braiding.
        \item $K=\langle u\rangle\cong \mathbb{Z}_2$ (the cyclic group generated by $u$), in which case we denote $\sRep(G,u)$ as $\sRep(W)$ the category of $\wedge{W}$-supermodules, for $\wedge{W}$ the exterior algebra of $W$. 
    \end{itemize}
     The exterior algebra has a structure of a Hopf superalgebra by setting $\Delta(w) = w\otimes 1+1\otimes w$, $\eps(w)=0$ and $S(w)=-w$ for $w\in W$.  

     The category $\sRep(W)$ has two simple objects - $\1$, and $S$ for which we have $S\otimes S = \1$ and $c_{S,S} = -\id$ for $c$ the symmetric structure of $\sRep(W)$. Notice that for $\dim(W)=0$, $\sRep(W)$ is simply $\sVec$, the category of super vector spaces.

     We prove the following lemma that we will use. This result is not new.
     \begin{lemma}\label{lemma: de-equivariantization of symmetric FTC}
         Let $G:=K\ltimes W$ be a finite supergroup with $u\in K$ central of order $2$, $N:=<u>\ltimes W$ and $\Gamma := K/<u>$. $\Rep(\Gamma)\subset \sRep(G,u)$ is a tannakian central subcategory and $\sRep(G,u)_\Gamma \cong \sRep(W)$ as braided categories.
     \end{lemma}
     \begin{proof}
         $N<G$ is a normal Hopf subalgebra with quotient $\kk[\Gamma]$, so we have an exact sequence of Hopf algebras 
         $$
            \kk\to N\to G\to \kk[\Gamma]\to \kk
         $$
         and the result follows from \cite{BN3}.
     \end{proof}
     \begin{remark}\label{remark: mod_C(A) for W}
         Notice that $A=\wedge{W^*}$ with the $W$-action given by contraction of $W$ with $W^*$ is a commutative algebra object in $\sRep(W)$ and $\Mod_{\sRep(W)}(A) = \sVec$.
     \end{remark}
     \cite{GS} classifies non-degenerate braided tensor categories which contain $\sRep(W)$ as a Lagrangian subcategory. These categories serve a key role in the structure of braided non-semisimple near-group categories.
     \begin{definition}
         A braided finite tensor category $\C$ is said to be  \textit{of purely odd Lagrangian type} if it is non-degenerate and contains a Lagrangian subcategory braided equivalent to $\sRep(W)$ for some purely odd super vector space $W\ne 0$. 
     \end{definition}
     For ease of read, we will refer to such categories as POL-categories.
\section{The structure of braided simple extensions}\label{sec: braided simple projective extensions}    
    Let $\C$ be a braided simple extension of $\D$ with braiding $c$, and assume that $\D$ is not semisimple. Let $Q$ be the unique simple projective object in $\C$ which is not in $\D$, and let $\{L_i\}_{i=1}^{n}$ be the (isomorphism classes of) simple objects in $\D$ with projective covers $P_i$. The existence of the simple projective object $Q$ and the braiding on $\C$ enables us to understand the structure of $\C$ and $\D$.

    \subsection{Basic properties of simple extensions}
    The following lemma is true for a general simple extension $\C$ of $\D$, not necessarily braided:
    \begin{lemma}\label{lemma: tensor product of simple projective object}
        Let $\D$ be a non semisimple finite tensor category, $\C$ a simple extension of $\D$, and let $Q$ be the simple projective object of $\C$. For a simple object $L\in \D$ with projective cover $P\in \D$ we have $[Q\otimes Q:P] = \FPdim_{\D}(L)$.
    \end{lemma}
    \begin{proof}
        The category $\D$ has a simple extension, hence by definition has a fiber functor $\D\to \Vect$. In particular, $\D$ is integral, and $\FPdim_{\D}(L) = \dim_{\kk}(L)$ (see \cite{EGNO}). $\M=<Q>$ is a $\D$-module category, and $Q$ is projective, hence we have an isomorphism $L\otimes{Q}\cong{\oplus_{i=1}^{\dim_{\kk}(L)}Q}$, which gives
        $$
            0\ne\oplus_{i=1}^{\dim_{\kk}(L)}\Hom_{\C}(Q,Q)={\Hom_{\C}(Q,L\otimes{Q})} = \Hom_{\C}(Q\otimes{Q^*},L) = \Hom_{\C}(Q\otimes{Q},L)
        $$
        the last equality arises from the fact that $Q$ is self-dual.

        Since $L$ is simple, each such nonzero morphism is epi, so we get a morphism $P\to Q\otimes{Q}$, which implies that $P$ is a direct summand of $Q\otimes Q$, because $Q\otimes{Q}$ is projective.

        By Schur's lemma $\dim(\Hom_{\C}(Q,Q))=1$, hence the multiplicity of $P$ in $Q\otimes{Q}$ is exactly $\dim_{\kk}(L)=\FPdim_{\D}(L)$.
\end{proof}
    By \cref{lemma: tensor product of simple projective object}, we have:
    \begin{equation}\label{eq: Q times Q}
        Q\otimes Q = \oplus_{i=1}^{n}\FPdim(L_i)P_i\oplus rQ
    \end{equation}
    for some $r\in \N$. By definition, we have $\FPdim(\C) = \FPdim(\D)+\FPdim(Q)^2$, so using \cref{eq: Q times Q} we get that $\FPdim(\C) = 2\FPdim(\D)+r\FPdim(Q)$. 
    
    \subsection{Degeneracy in braided simple extensions}
    From here on, let $\D$ be a braided non-semisimple finite tensor category, $\C$ a braided simple extension of $\D$, with a unique simple projective object $Q$. Let $\E_+\subset \D$ be the full  subcategory generated by objects $X$ which centralize $Q$.
    \begin{remark}\label{remark: scalar maps braiding on isomorphic objects}
        Assume we have $X,Y,Z\in \C$ such that $s_{X, Y} = \lambda\cdot \id_{X\otimes Y}$ for some $\lambda\in \kk^*$ and $\psi:X\simeq Z$, then $s_{Z, Y} = \lambda\cdot \id_{Z\otimes Y}$, by the following diagram:
        \begin{center}
            \begin{tikzcd}
                X\otimes Y \arrow[r, "s_{X, Y}"] \arrow[d, "\psi\otimes \id_{Y}"'] & X\otimes Y \arrow[d, "\psi\otimes\id_{Y}"] \\
                Z\otimes Y \arrow[r, "s_{Z, Y}"] & Z\otimes Y
            \end{tikzcd}
        \end{center}
        which commutes due to the naturality of the braiding $c$. 
    \end{remark}
    \begin{prop}\label{prop: the subcategory E+}
        $\E_+$ is a symmetric tensor subcategory of $\D$. 
    \end{prop}
    \begin{proof}
    Denoting by $\E$ the centralizer of the set $\{Q\}$ in $\C$, we have $\E_+ = \D\cap \E$, and both are tensor subcategories of $\C$, thus so is $\E_+$.
    
    Assume $X,Y$ centralize $Q$. By \cref{eq: squared braiding on (X times Y) times Z} for the triplet $(Q,X,Y)$ and $s_{Q,Y} = \id_{Q\otimes Y}$ we get:
    $$
        s_{Q\otimes X,Y} = a^{-1}_{Q,X,Y}\circ (\id_Q\otimes s_{X,Y})\circ a_{Q,X,Y},  
    $$
    while on the other hand, $Q\otimes X$ is a direct sum of copies of $Q$, so by \cref{remark: scalar maps braiding on isomorphic objects} $s_{Q\otimes X,Y} = \id_{Q\otimes X\otimes Y}$, which implies that $\id_Q\otimes s_{X,Y} = \id_{Q\otimes X\otimes Y}$, thus $s_{X,Y} = \id_{X\otimes Y}$. 
    \end{proof}
       
\begin{crl}\label{crl: If Q in D' then D symmetric}
        If $Q\in \D'$ then $\D = \E_+$ is symmetric.
    \end{crl}
    \begin{proof}
        By definition, if $Q\in \D'$ then $\E_{+} = \D$, and by \cref{prop: the subcategory E+}, $\D$ is symmetric.
    \end{proof}

        \begin{prop}\label{prop: squared braiding on Q}
        Let $\D$ be a non-semisimple braided finite tensor category, $\C$ a braided simple extension of $\D$ with a simple projective object $Q$, then $s_{Q,Q}\ne \id_{Q\otimes Q}$
    \end{prop}
    \begin{proof}
        Assume $s_{Q,Q}=\id_{Q\otimes Q}$, then by \cref{eq: squared braiding on (X times Y) times Z} for $X=Y=Z=Q$ we get $s_{Q\otimes Q,Q}=\id_{Q\otimes Q\otimes Q}$. By \cref{eq: Q times Q}, we get $s_{Q,P_i} = \id_{Q \otimes P_i}$, thus $P_i\in \E_+$ for each $i$. Every object in $\D$ is a subquotient of some direct sum of $\{P_i\}$, and $\E_+$ is closed under direct sums and subquotients, thus $\D\subset \E_+$.
        
        By \cref{crl: If Q in D' then D symmetric}, $\D$ is symmetric, and $Q$ centralizes $\C$ (as it centralizes $\D$ and $s_{Q,Q} = \id_{Q\otimes Q}$), thus $\C$ is symmetric. 

        The category $\C$ is a symmetric finite tensor category over an algebraically closed field of characteristic 0, thus by \cite{D}, $\C$ is equivalent to a category of representations of some finite supergroup, but such a category has no simple projective object, unless it is semisimple itself, thus contradicting our assumption.
    \end{proof}    
    
    \begin{thm}
        Let $\D$ be a braided finite tensor category. If $\D$ has a braided simple extension $\C$, then $\D$ is degenerate.
    \end{thm}
    \begin{proof}
        Denote $\C=\D\oplus \M$ and let $Q$ be the simple projective object in $\M$. Set $\FPdim(\D)=n$ and $\FPdim(Q)=q$ ($q\in \kk^*$). By \cref{eq: Q times Q}, $q^2=n+rq$, thus $\FPdim(\C)=2n+rq$. By \cref{eq: shimizu} we get in $\C$:
        \begin{equation}\label{eq: FPdim(D')}
            n\FPdim(\D') = (2n+rq)\FPdim(\C'\cap \D)
        \end{equation}
        which implies that $\FPdim(\D') > \FPdim(\C'\cap \D)$. If $Q\notin \D'$, then $\D'\subset \D$, hence a tensor subcategory of $\Z_2(\D)$, with $\FPdim(\D')>1$. This implies that $\Vect\subsetneq \Z_2(\D)$, thus $\D$ is degenerate. On the other hand, if $Q\in \D'$, then by \cref{crl: If Q in D' then D symmetric}, $\D$ is symmetric, hence degenerate.
    \end{proof}    
    \begin{prop}\label{prop: E_+ is central in C}
        $\C' = \E_+$. In particular, $\E_+$ is central in $\C$.
    \end{prop}
    \begin{proof}
        We have $\E_+\subset\D\subset\C$. We will apply \cref{eq: shimizu} for the pairs $(\E_+,\D)$ and $(\E_+,\C)$. We get:
        \begin{equation}\label{eq: shimizu for E and D}
            \FPdim(\E_+)\FPdim((\E_+)'_\D) = \FPdim(\D)\FPdim(\E_+\cap(\D)_\D'),
        \end{equation}
        \begin{equation}\label{eq: shimizu for E and C}
            \FPdim(\E_+)\FPdim(\E_+') = \FPdim(\C)\FPdim(\E_+\cap\C').
        \end{equation}
        On one hand, we have $\C'\cap \E_+ \subset (\D)_\D'\cap \E_+$, while on the other hand, for $X\in (\D)_\D'\cap \E_+$ we have $s_{X,Y}= \id_{X\otimes Y}$ for all $Y\in \D$ (because $X\in \D'$) and for $Y=Q$ (because $X\in \E_+$), thus $X\in \C'\cap \E_+$, hence $\C'\cap \E_+ = (\D)_\D'\cap\E_+$, so $\FPdim(\E_+\cap(\D)_\D') = \FPdim(\E_+\cap\C')$. This gives the equality:
        \begin{equation}\label{eq: fractions of FPdims}
            \frac{\FPdim(\D)}{\FPdim((\E_+)'_\D)} = \frac{\FPdim(\C)}{\FPdim(\E_+')}
        \end{equation}
        Let us denote $\FPdim(\D) = n$, $\FPdim(Q)=q$, $\FPdim(\E'_+) = e_c$, $\FPdim((\E_+)'_\D) = e_d$. By \cref{eq: Q times Q} we have $\FPdim(\C) = 2n+rq$. In addition, by definition $Q\in \E'_+$, thus by calculation of Frobenius-Perron dimension we get $e_c = e_d + q^2 = e_d + n+rq$. Substituting these in \cref{eq: fractions of FPdims} gives:
        $$
            \frac{n}{e_d} = \frac{2n+rq}{e_d+n+rq}, 
        $$
        which translates to the quadratic equation:
        $$
            n^2+n(rq-e_d)-e_drq = 0.
        $$
        This quadratic equation has two solutions: $n=e_d$ or $n=-rq$. 
        The category $\D$ is integral (it has a fiber functor), and $r,q$ are non-negative\footnote{Notice that $q$ is not necessarily integer, but in that case as well $n=-rq$ is not a valid solution, because $n$ and $r$ are integers.}, so we must have $n=e_d$. This implies that: $$\FPdim(\E'_+) =  e_d+n+rq=2n+rq = \FPdim(\C),$$
        thus $\E'_+ = \C$, which in turn implies that $\E_+\subseteq \C'$. On the other hand, by \cref{prop: squared braiding on Q}, $Q\notin \C'$, thus $\C'\subseteq\D$. This implies that for any $X\in \C'$, $X\in \D$ and $s_{Q,X} = \id_{Q\otimes X}$, but this means by definition that $X\in \E_+$, thus $\C'\subseteq \E_+$, completing the proof.
    \end{proof}
    \begin{remark}
        Notice that in general, we don't have $\E_+ = \D'$. For instance, when $\D$ is pointed and $\C$ is non-degenerate, we will prove that $\D$ is slightly degenerate, thus $\E_+ = \Vect$ while $\D' = \sVec$ (see \cref{thm: exact sequence of braided near-group categories}(1) in \cref{sec: simple extensions of pointed FTC} with $K$ trivial). 
    \end{remark}
        \subsection{The absence of braided simple extension of non-semisimple symmetric categories}\label{sec: no extensions to symmetric categories}
        In this subsection we prove that for a braided simple extension $(\C,\D,\M)$ for $\D$ non semisimple, $\D\not\subset\Z_2(\C)$. The proof is based on \cref{prop: sRep(W)+M}, which concerns braided finite tensor categories containing a central subcategory braid equivalent to $\sRep(W)$ (with the symmetric braiding). 
        \subsection*{The strategy of the proof} Suppose $\D\subset\Z_2(C)$. The argument passes through three categories, each obtained from the previous one, and follows a simple projective object along the way:
        \begin{enumerate}
            \item \textit{From $\C$ to $\E$} (\cref{prop: reduction of D central to sRep(W)}). De-equivariantizing $\C$ by the Tannakian part of $\D$ gives a category $\E=\E_0\oplus\T$ with $\E_0\simeq\sRep(W)$ central and $\T$ semisimple. The image of $Q$ is non-zero, so $\T$ contains a simple object $X$, which is projective in $\E$.
            \item \textit{From $\E$ to $\B$} (\cref{lemma: Structure and properties of B}). Taking modules over the algebra $A=\wedge W^*\in\E_0$ gives a braided \emph{fusion} category $\B=\Mod_\E(A)$, and $F(X)=X\otimes A$ is a non-zero object of $\B$. Since $\B$ is semisimple, $F(X)$ alone carries no useful information.
            \item \textit{From $\E$ to $\Mod_\B(\Lambda_\B)$} (proof of \cref{prop: sRep(W)+M}). We equip $F(X)$ with additional structure: an action $x$ of the exterior algebra $\Lambda_\B=\wedge W$, viewed as a Hopf algebra in $\B$. This structure records the part of $X$ that $F$ forgets.
            \item \textit{Contradiction.} The module $(F(X),x)$ is simple and projective, because $X$ is simple and projective in $\E$. But $\Mod_\B(\Lambda_\B)$ has no simple projective objects (Lemma~\ref{lem:lambda}).
        \end{enumerate}
        The action $x$ in step (3) is obtained as follows. The Hopf algebra $\wedge W$ acts on $A$ (by contractions), and we consider actions of $\wedge W$ on the underlying object of an $A$-module $M$ for which the action map $M\otimes A\to M$ is $\wedge W$-linear; equivalently, module structures over the smash product $A\rtimes\wedge W$ extending the $A$-module structure. The object $F(X)$ carries two such actions: the free one, $\nabla^X$ ($\wedge W$ acting on the factor $A$ only), which remembers that $F(X)$ comes from $\E$; and a canonical one, $\nabla^0$, which exists on every object of $\B$ and is compatible with all morphisms and tensor products of $\B$. Neither is $A$-linear, but their difference is, and this difference is the $\Lambda_\B$-action $x$. \cref{lemma: space of derivations},\cref{lem:hom} and \cref{lem:simple} set up the proper derivations and relate them to $\E$; \cref{lem:canonical}, \cref{lem:monoidal} and \cref{lem:flat} construct $\nabla^0$, and this is where centrality of $\E_0$ is used.

        To prove this result, we will use the following three facts concerning a fusion category $\B$:
        \begin{enumerate}
            \item \textbf{Universal grading} \cite[4.14]{EGNO}: $\B$ is faithfully graded by a finite group $U(\B)$ with trivial component $\B_{ad}$, the subcategory generated by the summands of $M\otimes M^*$; in particular, for every simple $M$, all summands of $M^{\otimes|U(\B)|}$ lie in $\B_{ad}$.
            \item \cite{DGNO}: if $S$ is an invertible object of a braided fusion category with $S\otimes S\cong\mathbf 1$, $c_{S,S}=-\id$, and $S$ centralizes a simple object $M$, then $S\otimes M\not\cong M$. 
            \item For a fusion category $\B$ and a $Z\in \Z_2(\B)$, let $C^m(\B;Z)$ be the space of natural transformations $\theta_{X_1,\dots,X_m}: X_1\otimes\cdots\otimes X_m\to Z\otimes X_1\otimes\cdots\otimes X_m$, with
            $$
            (\delta\theta)_{X_0,\dots,X_m}=\id_{X_0}\otimes\theta_{X_1,\dots,X_m}+\sum_{i=1}^{m}(-1)^i\theta_{X_0,\dots,X_{i-1}\otimes X_i,\dots,X_m}+(-1)^{m+1}\theta_{X_0,\dots,X_{m-1}}\otimes\id_{X_m},
            $$
            where $Z$ is moved to the left using the braiding. This complex is called \textit{Davydov-Yetter complex with coefficients}. By \cite[Corollary 3.18]{GHS}, over a field of characteristic $0$, $H^m(\B;Z)=0$ for $m\geq1$. 
        \end{enumerate}
        \begin{prop}\label{prop: reduction of D central to sRep(W)}
            Let $(\C,\D,\M)$ be a braided simple extension such that $\D$ is a non semisimple symmetric finite tensor category such that $\D\subset \Z_2(\C)$, then we have a braided tensor category $\E$ with abelian decomposition $\E = \E_0\oplus \T$ such that $\E_0\cong \sRep(W)\subset \Z_2(\E)$, $\T$ is semisimple, $\T\neq 0$ and $\E_0\otimes \T\subset \T$.
        \end{prop}
        \begin{proof}
        Suppose $\D\subset\Z_2(\C)$. Then $\D$ is a finite symmetric tensor category, so by Deligne's theorem \cite{D}, $\D\simeq\sRep(K\ltimes W,u)$, with $K$ a finite group, and $W\neq0$ since $\D$ is not semisimple. Put $N=\langle u\rangle\ltimes W$ and $\Gamma=K/\langle u\rangle$. The objects of $\D$ on which $N$ acts trivially form a Tannakian subcategory $\Rep(\Gamma)\subset\D\subset\Z_2(\C)$. Let $\E:=\C_\Gamma$ and let $F:\C\to\E$ be the free module functor. Then:
        \begin{enumerate}
            \item $\E_0:=\D_\Gamma\simeq\sRep(N,u)=\sRep(W)$ by \cref{lemma: de-equivariantization of symmetric FTC}
            \item $\E=\E_0\oplus\T$ as abelian categories, where $\T$ is the category of $\fun(\Gamma,\kk)$-modules in $\M$; indeed $\fun(\Gamma,\kk)\in\D$ and $\D\otimes\M\subset\M$, so the module structure preserves the decomposition $\C=\D\oplus\M$;
            \item $\T$ is semisimple, since $\M$ is semisimple and $\fun(\Gamma,\kk)$ is separable, so every short exact sequence of $A$-modules in $\C$ that splits in $\C$ also splits in $\Mod_\C(A)$; and $\T\neq0$ because $F(Q)\neq0$;
            \item $\E_0\subset\Z_2(\E)$, since $F$ is a braided tensor functor and every object of $\E$ is a quotient of an object of the form $F(Y)$.
        \end{enumerate}
    \end{proof}

        \begin{prop}\label{prop: sRep(W)+M}
            Assume $\E$ is a braided finite tensor category such that $\E = \E_0\oplus \T$ as abelian categories, with $\E_0\cong \sRep(W)$ such that $\dim(W)\ne 0$, $\E_0\subset \Z_2(\E)$,  $\T$ semisimple and $\E_0\otimes \T\subset \T$, then $\T=0$.
        \end{prop}
        This proposition together with the previous one give the desired contradiction. We prove this proposition with a series of lemmas. We fix the following notation:
        \begin{itemize}
            \item $n=\dim(W)$, $S$ is the non-trivial invertible object in $\sRep(W)$, so $S\otimes S = \1$ and $c_{S,S} = -\id$. $S\in\Z_2(\E)$, so the braiding gives a canonical way of moving $S$ past any object of $\E$; we use it, together with the Koszul sign $-1$ whenever two copies of $S$ are interchanged.
            \item $A = \wedge{W^*}$, with the $W$-action given by the contraction $\{\iota_w\}_{w\in W}$ ($w\cdot u^* = u^*(w)$ for $u,w\in W$ and extended by Leibniz rule); $\iota_w$ is a morphism $A\to S\otimes A$, $\iota_w\iota_{w'}+\iota_{w'}\iota_w=0$, and $\bigcap_{w}\ker\iota_w$ is the image of the unit $\eta:\1\to A$.
            \item $\B = \Mod_\E(A)$ with $F$ the free module functor and $U$ the forgetful functor $U:\B\to \E$, $S_\B = F(S)$. The tensor product in $\B$ is $\otimes_A$.
        \end{itemize}
        \begin{lemma}\label{lemma: Structure and properties of B}
            $\B$ is a braided fusion category and $F:\E\to\B$ is a braided tensor functor. The object $S_\B$ is invertible and transparent in $\B$, $S_\B\otimes_A S_\B\cong A$ and $c_{S_\B,S_\B}=-\id$. Moreover:
            \begin{enumerate}
                \item $\Hom_\B(M,S_\B\otimes_A M)=0$ for every simple object $M\in\B$;
                \item every simple object of $\B$ is a direct summand of an object $F(Y)$, $Y\in\E$;
                \item $\Mod_{\E_0}(A)\subset\B$ is a fusion subcategory equivalent to $\sVec$, generated by $S_\B$.
            \end{enumerate}
 
        \end{lemma}
        \begin{proof}
            $A$ is a commutative algebra in $\E_0$ such that $\Hom_\E(\1,A) = \kk$, and $\Mod_{\E_0}(A)\simeq\sVec$. In particular $A$ is exact in $\E_0$, hence in $\E$ (\cite{SY}), and since $A\in\Z_2(\E)$, $\B$ is a braided finite tensor category and $F$ is a braided tensor functor by \cref{prop: de-equivariantizaion of braided finite tensor categories}. Its unit object $A$ lies in $\Mod_{\E_0}(A)\simeq\sVec$ and is therefore projective in $\B$ (as $A$ preserves the decomposition $\E=\E_0\oplus\T$, so does $\B$), hence $\B$ is semisimple. The statements on $S_\B$ follow from $S\in\Z_2(\E)$, $S\otimes S\cong\1$ and $c_{S,S}=-\id$, since $F$ is braided and every object of $\B$ is a quotient of some $F(Y)$. Part~(1) follows by \cite{DGNO}, which applies since $S_\B$ is transparent. Part (2): the action $U(M)\otimes A\to M$ is an epimorphism in $\B$, hence split. Part (3) is \cref{remark: mod_C(A) for W}.
        \end{proof}
    \subsection*{$\iota$-derivations}
        Let $M\in\B$. A \textit{$\iota$-derivation} on $M$ is a family $\nabla=(\nabla_w)_{w\in W}$, linear in $w$, of morphisms $\nabla_w: U(M)\to S\otimes U(M)$ in $\E$ such that
        \begin{equation}\label{eq: i-derivation}
            \nabla_w(m\cdot a)=\nabla_w(m)\cdot a+(-1)^{|m|}\,m\cdot\iota_w(a)
        \end{equation}
        (in diagrammatic form: $\nabla_w\circ\mu_M=(\id_S\otimes \mu_M)\circ(\nabla_w\otimes\id_A)+(\id_S\otimes \mu_M)\circ(\beta\otimes \id_A)\circ(\id_M\otimes\iota_w)$, where $\mu_M$ is the action, $\beta:M\otimes S\to S\otimes M$ the braiding). Its \textit{curvature} is $\kappa^\nabla(w,w'):=\nabla_w\nabla_{w'}+\nabla_{w'}\nabla_w$, composed using $S\otimes S\cong\1$; $\nabla$ is \textit{flat} if $\kappa^\nabla=0$. Thus a flat $\iota$-derivation is the same as an action of $\wedge W$ on $U(M)$, with each $w$ acting by an odd map (written as an even morphism into $S\otimes U(M)$), for which the action map $U(M)\otimes A\to U(M)$ is $\wedge W$-linear; these are the actions described in the strategy above. The term \textit{$\iota$-derivation} refers to the general version, which we need because flatness of the canonical operator is only established in Lemma~\ref{lem:flat}. A $\B$-morphism $\varphi: M\to M'$ \textit{commutes} with $\iota$-derivations $\nabla,\nabla'$ if $\forall w~~\nabla'_w\varphi=(\id_S\otimes\varphi)\nabla_w$, and a $\B$-subobject $K\subset M$ is \textit{$\nabla$-stable} if $\forall w~~\nabla_w(U(K))\subset S\otimes U(K)$.

        For $X\in\E$, the \textit{free derivation} $\nabla^X:=\id_X\otimes\iota$ on $F(X)$ is a flat $\iota$-derivation. If $\nabla,\nabla'$ are $\iota$-derivations on $M,M'$, then $\nabla\otimes\id+\id\otimes\nabla'$ on $U(M)\otimes U(M')$ descends to a $\iota$-derivation $\nabla\boxtimes\nabla'$ on $M\otimes_A M'$, by \cref{eq: i-derivation}; under the canonical isomorphism $F(X)\otimes_AF(Y)\cong F(X\otimes Y)$ one has $\nabla^X\boxtimes\nabla^Y=\nabla^{X\otimes Y}$. Finally, the difference of two $\iota$-derivations on $M$ is $A$-linear, since the last term in \cref{eq: i-derivation} cancels. Hence:

        \begin{lemma}\label{lemma: space of derivations}
            The set of $\iota$-derivations on $M\in\B$ is either empty or an affine space over $\Hom_\B(M,S_B\otimes_AM)\otimes W^*$.\qed
        \end{lemma}

        The next two lemmas describe how much of $\E$ can be recovered from $\B$ together with the free operators.

        \begin{lemma}\label{lem:hom}
            Let $X,Y\in\E$.
            \begin{enumerate}
                \item $\bigcap_w\ker\nabla^X_w=X\otimes\mathbf 1\subset X\otimes A$.
                \item A $\B$-morphism $\varphi\colon F(X)\to F(Y)$ commutes with $\nabla^X$ and $\nabla^Y$ if and only if $\varphi=F(\varphi_0)$ for some morphism $\varphi_0\colon X\to Y$ in $\E$.
            \end{enumerate}
        \end{lemma}

        \begin{proof}
        \begin{enumerate}
            \item Tensoring with $X$ is exact, so $\bigcap_w\ker(\id_X\otimes\iota_w)=X\otimes\bigcap_w\ker\iota_w=X\otimes\mathbf 1$.
            \item By the adjunction between the free module functor and the forgetful functor \cref{subsection: de-equivariantization}, $\varphi=\mu_{F(Y)}\circ(\hat\varphi\otimes\id_A)$ for a unique $\hat\varphi\colon X\to Y\otimes A$, namely $\hat\varphi=\varphi\circ(\id_X\otimes\eta)$. Since $\iota_w$ is a derivation of $A$, a direct computation gives
            $$
                \nabla^Y_w\circ\varphi-(\id_S\otimes\varphi)\circ\nabla^X_w=\mu_{F(Y)}\circ\big(((\id_Y\otimes\iota_w)\hat\varphi)\otimes\id_A\big),
            $$
            which is the morphism corresponding under the adjunction to the morphism $(\id_Y\otimes\iota_w)\hat\varphi\colon X\to S\otimes Y\otimes A$. Hence $\varphi$ commutes with the free derivations if and only if $\hat\varphi$ factors through $\bigcap_w\ker(\id_Y\otimes\iota_w)=Y\otimes\mathbf 1$ (part~(1)), i.e.\ $\hat\varphi=(\id_Y\otimes\eta)\varphi_0$, i.e.\ $\varphi=F(\varphi_0)$.
        \end{enumerate}
        \end{proof}

        \begin{lemma}\label{lem:simple}
            Let $X\in\E$ be simple. Then the only $\nabla^X$-stable $\B$-subobjects of $F(X)$ are $0$ and $F(X)$.
        \end{lemma}
        \begin{proof}
            Let $K\subset F(X)$ be a non-zero $\nabla^X$-stable subobject, and let $K^\nabla:=U(K)\cap(X\otimes\mathbf 1)$, which by \cref{lem:hom}(1) is the common kernel of the restrictions $\nabla^X_w|_{U(K)}: U(K)\to S\otimes U(K)$. If $K^\nabla=0$, then $(\nabla^X_w|_{U(K)})_w\colon U(K)\to W^*\otimes S\otimes U(K)$ is a monomorphism, and so is the $(n+1)$-fold iterate $U(K)\to(W^*\otimes S)^{\otimes(n+1)}\otimes U(K)$, since tensoring is exact. But this iterate vanishes, because $\nabla^X=\id_X\otimes\iota$ lowers the degree in $\wedge W^*$. Thus $U(K)=0$, a contradiction. Hence $K^\nabla$ is a non-zero subobject of $X\otimes\mathbf 1\cong X$, so $K^\nabla=X\otimes\mathbf 1$ as $X$ is simple. Since $K$ is an $A$-submodule, it contains $(X\otimes\mathbf 1)\cdot A=F(X)$.
        \end{proof}

        \begin{lemma}\label{lem:canonical}
            Every simple object $M\in\B$ carries exactly one $\iota$-derivation $\nabla^0_M$. Extending additively, every $M\in\B$ carries a $\iota$-derivation $\nabla^0_M$, independent of the chosen decomposition into simple objects, such that every $\B$-morphism commutes with $\nabla^0$. In particular every $\B$-subobject of $M$ is $\nabla^0_M$-stable, and $\nabla^0_A=\iota=\nabla^0_{S_\B}$.
        \end{lemma}
        \begin{proof}
            Let $M$ be simple. By \cref{lemma: Structure and properties of B}(2) there are $\B$-morphisms $i\colon M\to F(Y)$ and $p\colon F(Y)\to M$ with $pi=\id_M$, and $(\id_S\otimes p)\circ\nabla^Y\circ i$ is a $\iota$-derivation on $M$, as $p$ and $i$ are $A$-linear. Uniqueness follows from \cref{lemma: space of derivations} and \cref{lemma: Structure and properties of B}(1). For arbitrary $M$, write $M\cong\bigoplus_jM_j\otimes V_j$ with $M_j$ simple and pairwise non-isomorphic, and set $\nabla^0_M:=\bigoplus_j\nabla^0_{M_j}\otimes\id_{V_j}$. This commutes with $\mathrm{Aut}_\B(M)=\prod_jGL(V_j)$, so it does not depend on the decomposition, and by Schur's lemma every $\B$-morphism commutes with $\nabla^0$. Applying this to $K\hookrightarrow M$ gives stability of subobjects. Finally $A=F(\mathbf 1)$ and $S_\B=F(S)$ are simple and $\iota$ is a $\iota$-derivation on both.
        \end{proof}

        \begin{lemma}\label{lem:monoidal}
        $\nabla^0_{M\otimes_AN}=\nabla^0_M\boxtimes\nabla^0_N$ for all $M,N\in\B$.
        \end{lemma}
        \begin{proof}
            Fix $w\in W$ and set $\Theta_{M,N}:=(\nabla^0_M\boxtimes\nabla^0_N)_w-(\nabla^0_{M\otimes_AN})_w$. By \cref{lemma: space of derivations}, $\Theta_{M,N}\colon M\otimes_AN\to S_\B\otimes_AM\otimes_AN$ is a morphism in $\B$, and it is natural in $M$ and $N$ by Lemma~\ref{lem:canonical}. Computing $(\nabla^0_M\boxtimes\nabla^0_N\boxtimes\nabla^0_P)-\nabla^0_{M\otimes N\otimes P}$ in two ways, by inserting either $\nabla^0_{M\otimes N}$ or $\nabla^0_{N\otimes P}$, yields
            $$
                \id_M\otimes\Theta_{N,P}-\Theta_{M\otimes N,P}+\Theta_{M,N\otimes P}-\Theta_{M,N}\otimes\id_P=0,
            $$  
            i.e.\ $\Theta$ is a $2$-cocycle in the Davydov--Yetter complex of $\B$ with coefficients in $S_\B$, thus by \cite[Corollary 3.18]{GHS}, $\Theta=\delta\theta$ for some $1$-cochain $\theta\in\operatorname{Nat}(\id_\B,S_\B\otimes_A\id_\B)$. But this space is $\bigoplus_{M}\Hom_\B(M,S_\B\otimes_AM)=0$, the sum over simple objects $M$, by \cref{lemma: Structure and properties of B}(1). Hence $\Theta=0$.
        \end{proof}

        \begin{lemma}\label{lem:flat}
            $\nabla^0_M$ is flat for every $M\in\B$.
        \end{lemma}

        \begin{proof}
            For any $\iota$-derivation $\nabla$ on $M$, the curvature $\kappa^\nabla(w,w')$ is $A$-linear: expanding with \cref{eq: i-derivation}, the terms involving $\iota$ cancel, since $\iota_w\iota_{w'}+\iota_{w'}\iota_w=0$ and the two mixed terms carry opposite Koszul signs. Hence, for $M$ simple, $\kappa_M:=\kappa^{\nabla^0_M}$ takes values in $\Hom_\B(M,S_\B\otimes_A S_\B\otimes_AM)=\End_\B(M)=\kk$, so $\kappa_M\in\operatorname{Sym}^2W^*$.

            The curvature of $\nabla\boxtimes\nabla'$ is $\kappa^\nabla\otimes\id+\id\otimes\kappa^{\nabla'}$, because odd operators acting on different tensor factors anticommute. By \cref{lem:monoidal} and \cref{lem:canonical}, the restriction of $\nabla^0_M\boxtimes\nabla^0_{M'}=\nabla^0_{M\otimes_AM'}$ to a simple summand $N$ of $M\otimes_AM'$ is $\nabla^0_N$. Hence $\kappa_N=\kappa_M+\kappa_{M'}$ for all simple $M,M'$ and all simple summands $N\subset M\otimes_AM'$.

            Now $\kappa_A=0$ since $\iota$ is flat. Taking $N=A\subset M\otimes_AM^*$ gives $\kappa_{M^*}=-\kappa_M$, so $\kappa$ vanishes on all simple objects of the adjoint subcategory $\B_{ad}$. Let $U(\B)$ be the universal grading group of $\B$; it is finite, and for every simple $M$ all simple summands of $M^{\otimes_A|U(\B)|}$ lie in $\B_{ad}$. Hence $|U(\B)|\cdot\kappa_M=0$, so $\kappa_M=0$ as $\operatorname{char}\kk=0$. Flatness for arbitrary $M$ follows since $\nabla^0_M$ is a direct sum of operators on simple objects.
        \end{proof}

        \subsection*{Modules over $\Lambda_\B$} Let $\Lambda_\B$ be the exterior algebra on $W$ (with $W$ odd), regarded as a Hopf algebra in the subcategory $\Mod_{\E_0}(A)\simeq\sVec$ of $\B$, with $w$ primitive. Since $\sVec\subset\Z_2(\B)$, $\Lambda_\B$-modules in $\B$ form a tensor category $\Mod_\B(\Lambda_\B)$. Concretely, a $\Lambda_\B$-module is a pair $(M,x)$ with $M\in\B$ and $x=(x_w)_{w\in W}$, linear in $w$, a family of $\B$-morphisms $x_w\colon M\to S_\B\otimes_AM$ with $x_wx_{w'}+x_{w'}x_w=0$; the tensor product is $(M,x)\otimes(N,y)=(M\otimes_AN,\,x\otimes\id+\id\otimes y)$.

        \begin{lemma}\label{lem:lambda}
            Assume $W\neq0$.
        \begin{enumerate}
            \item The simple objects of $\Mod_\B(\Lambda_\B)$ are the pairs $(M,0)$ with $M$ simple in $\B$.
            \item For every $N\in\Mod_\B(\Lambda_\B)$, the module $\Lambda_\B\otimes N$ is projective.
            \item $\Mod_\B(\Lambda_\B)$ has no simple projective objects.
        \end{enumerate}
        \end{lemma}

        \begin{proof}
            \begin{enumerate}
                \item If $(M,x)$ is simple, the image of the action of the augmentation ideal $\Lambda^+=\wedge^{\geq1}W$ is a submodule different from $M$ (otherwise $M=(\Lambda^+)^{n+1}M=0$), hence zero; so $x=0$ and $M$ is simple in $\B$.
                \item The map $a\otimes m\mapsto\sum a_{(1)}\otimes \tilde{S}(a_{(2)})m$ (for $\tilde{S}$ the antipode map) is an isomorphism $\Lambda_\B \otimes N\cong\Lambda_\B \otimes N_{\triv}$ (it is well defined because $\Lambda_\B$ is transparent), where $N_{\triv}$ is $N$ with the zero action. Free modules $\Lambda_\B\otimes N_{\triv}$ are projective, since $\Hom_{\Lambda_\B}(\Lambda_\B\otimes N_{\triv},-)\cong\Hom_\B(N,-)$ and $\B$ is semisimple.
                \item The projective cover of $(M,0)$, $M$ simple, is $\Lambda_\B\otimes(M,0)$: its endomorphism algebra is $\Hom_\B(M,\Lambda_\B\otimes M)\cong\wedge^{\mathrm{even}}W$ by \cref{lemma: Structure and properties of B}(1), which is local. If $(M,0)$ were projective it would be isomorphic to its projective cover, so $\FPdim(M)=2^n\FPdim(M)$, i.e.\ $n=0$.
            \end{enumerate}
        \end{proof}

        \begin{proof}[Proof of \cref{prop: sRep(W)+M}] 
        Suppose $\T\neq0$ and let $X\in\T$ be simple. Since $\T$ is a semisimple direct summand of $\E$, $X$ is simple and projective in $\E$. Put
        $$
            x:=\nabla^X-\nabla^0_{F(X)}.
        $$
        By \cref{lemma: space of derivations}, each $x_w$ is a $\B$-morphism $F(X)\to S_\B\otimes_AF(X)$.

        \smallskip\noindent\emph{Step 1: $(F(X),x)$ is a $\Lambda_\B$-module.} Expanding $0=\kappa^{\nabla^X}$ with $\nabla^X=\nabla^0+x$ and using $\kappa^{\nabla^0}=0$ (\cref{lem:flat}), we get
        $$
            0=\big(\nabla^0_wx_{w'}+x_{w'}\nabla^0_w\big)+\big(\nabla^0_{w'}x_w+x_w\nabla^0_{w'}\big)+\big(x_wx_{w'}+x_{w'}x_w\big).
        $$
        By \cref{lem:canonical} applied to the $\B$-morphism $x_{w'}$, and by \cref{lem:monoidal} together with $\nabla^0_{S_\B}=\iota$, which give $\nabla^0_{S_\B\otimes_AF(X)}=-(\id_S\otimes\nabla^0_{F(X)})$ on $S\otimes U(F(X))$ (the sign arises from moving the odd operator past $S$), the first two brackets vanish. Hence $x_wx_{w'}+x_{w'}x_w=0$.

        \smallskip\noindent\emph{Step 2: $(F(X),x)$ is simple.} A $\Lambda_\B$-submodule is a $\B$-subobject $K\subset F(X)$ which is $x$-stable. Since $K$ is $\nabla^0$-stable (\cref{lem:canonical}), it is $x$-stable if and only if it is $\nabla^X$-stable, so $K\in\{0,F(X)\}$ by \cref{lem:simple}.

        \smallskip\noindent\emph{Step 3: $(F(X),x)$ is projective.} Let $P\in\E_0$ be the projective cover of $\mathbf 1$ (the regular $\wedge W$-module), with $\pi\colon P\to\mathbf 1$. Since $X$ is projective, $\pi\otimes\id_X\colon P\otimes X\to X$ has a section $\sigma$. The morphisms $F(\sigma)$ and $F(\pi\otimes\id_X)$ commute with the free operators (\cref{lem:hom}(2)) and with $\nabla^0$ (\cref{lem:canonical}), hence with $x$; so $(F(X),x)$ is a direct summand of $(F(P\otimes X),x_{P\otimes X})$, where $x_{P\otimes X}=\nabla^{P\otimes X}-\nabla^0$. Since $\nabla^{P\otimes X}=\nabla^P\boxtimes\nabla^X$ and $\nabla^0$ is monoidal (\cref{lem:monoidal}),
        $$
            (F(P\otimes X),x_{P\otimes X})\cong(F(P),x_P)\otimes(F(X),x)
        $$
        as $\Lambda_\B$-modules. Under $\Mod_{\E_0}(A)\simeq\sVec$, $F(P)$ corresponds to the underlying super vector space of $P$, and $x_P$ to the action of $W$ on $P$; hence $(F(P),x_P)\cong\Lambda_\B$ is the regular module, and $(F(X),x)$ is a direct summand of $\Lambda_\B\otimes(F(X),x)$, which is projective by \cref{lem:lambda}(2).

        \smallskip
        Steps~2 and~3 contradict \cref{lem:lambda}(3). Hence $\T=0$.
    \end{proof}
    Finally, we can prove our main result of this subsection:
    \begin{prop}\label{prop: D is not central in C}
        Let $(\C,\D,\M)$ be a braided simple extension such that $\D$ is non semisimple, then $\D\not\subset\Z_2(\C)$.    
    \end{prop}
    \begin{proof}
        Assume $\D\subset\Z_2(\C)$. By \cref{prop: reduction of D central to sRep(W)}, there exists a braided tensor category $\E$ which decompose as an abelian category to $\sRep(W)\oplus \T$ for $W$ a purely odd super vector space such that $\dim(W)\ne 0$, $\T\ne 0$ a semisimple category, $\sRep(W)\subset \Z_2(\E)$ and $\sRep(W)\otimes \T\subset \T$. On the other hand, by \cref{prop: sRep(W)+M}, in this case we must have $\T = 0$, which is a contradiction.
    \end{proof}
    
    \section{Braided simple extensions of pointed tensor categories}\label{sec: simple extensions of pointed FTC}
    In this section, we assume that $\D$ is a pointed braided finite tensor category, and $\C$ is a braided simple extension of $\D$. Concrete examples of such categories can be found in Appendix A in \cite{GS}. Denote the simple projective object in $\C$ by $Q$. Since $\D$ is pointed, the simple objects of $\D$ generate a pointed braided fusion subcategory, which we denote by $\D_f$. This category corresponds to a pair $(G,q)$ for $G$ a finite abelian group and $q$ a quadratic form on $G$. The group $G$ is simply $\pic(\D)$ (the Picard group), and the squared braiding $s_{L_g,L_h}:L_g\otimes L_h\to L_h\otimes L_g$ of two simple objects $L_g,L_h\in \D_f$ is multiplication by $b(g,h) = \frac{q(gh)}{q(g)q(h)}$ (see for instance \cite[8.4]{EGNO}).

    From here on, we denote the simple objects of $\D$ by $L_g$ ($g\in G$). We denote the projective cover of $L_g$ by $P_g$, and set $\bar{P} := \oplus_{g\in \pic(\D)}P_g$. Clearly we have $\bar{P}\otimes L_g \cong \bar{P}$ for all $g\in G$. 
    
    \subsection{The \texorpdfstring{$Q$}{Q}-braiding group of \texorpdfstring{$\C$}{C}}\label{sec: Q-grading of Pic(D)} By the definition of a simple extension, we have $$s_{Q,L_g}\in \Hom_{\C}(Q\otimes L_g, Q\otimes L_g) = \Hom_{\C}(Q,Q).$$
    Since $\kk$ is algebraically closed, we get by Schur's lemma that $s_{Q,L_g}$ is simply multiplication by a scalar, which we denote by $\lambda_g$. 

    Applying \cref{eq: squared braiding on X times (Y times Z)} to the triplet $(Q,L_g,L_h)$, and using the fact that $s_{Q,L_h}$ are scalars, we get that $\lambda_g\lambda_h = \lambda_{gh}$, and $s_{Q,\1} = 1$. In other words, we have a group homomorphism $\psi:\pic(\D)\to \kk^*$ given by $\psi(L_g)=\lambda_g$. We define $G_Q:=\operatorname{Im}(\psi)$ and call this the \textit{$Q$-braiding group of $\C$}.
    
    \begin{prop}\label{prop: values of lambda_g}
        We have the following properties:
        \begin{enumerate}
            \item $\lambda_g^2 = 1$ for all $g\in G$.
            \item $L_g\in \Z_2(\D)$ for all $g\in G$.
            \item $\ker(\psi) = \pic(\Z_2(\C))$.
        \end{enumerate}
    \end{prop}
    \begin{proof}
        \begin{enumerate}
            \item Applying \cref{eq: squared braiding on (X times Y) times Z} for the triplet $\{Q,Q,L_g\}$ gives $s_{Q\otimes Q,L_g} = \lambda_g^2\cdot \id_{Q\otimes Q}$. In particular, for $P_\1$ as a direct summand of $Q\otimes Q$ we get $s_{P_\1,L_g} = \lambda^2_g\id_{P_{\1}}$. Using the epimorphism $P_{\1}\to \1$ and $S_{\1,L_g} = \id_{L_g}$, the claim follows by naturality of the squared braiding.
            \item Using the same argument as above, we get that $s_{P_h,L_g} = \lambda_g^2\id_{P_h\otimes L_g} = \id_{P_h\otimes L_g}$. Every object in $\D$ is a subquotient of direct sums of the different $P_h$'s, so by naturality the claim follows.
            \item By definition $\lambda_g = 1$ exactly for $g$ such that $L_g\in \E_+$, thus the claim follows from \cref{prop: E_+ is central in C}.
        \end{enumerate}
    \end{proof}
    In a way, the group $G_Q$ "encodes" the additional braiding data of the simple extension $\C$ of $\D$. 

    \subsection{Non-semisimple braided near-group categories}    
    \begin{definition}
        Let $\D$ be a pointed non-semisimple finite tensor category admitting a fiber functor, and let $G=\pic(\D)$. 
        \begin{enumerate}
            \item A simple extension $\C$ of $\D$ is called a \textit{non-semisimple near-group category}.
            \item Denoting the unique simple projective object of $\C$ by $Q$, we say that $\C$ has a \textit{generalized fusion rule $(G,r)$} for $r$ such that \cref{eq: Q times Q} holds.
            \item   If $\D$ and $\C$ are braided, and the restriction of the braiding of $\C$ to $\D$ coincide with the braiding of $\D$, we say that $\C$ is a \textit{braided} non-semisimple near-group category.
        \end{enumerate}    
    \end{definition}
    As abelian categories, $\C = \D\oplus \M$ for $\M$ the semisimple abelian category generated by $Q$, so the functor $-\otimes Q$ on $\D$ is given by $X\otimes Q = Q\otimes \omega_Q(X)$ for $\omega_Q(X) = \Hom(Q,X\otimes Q)$, and $\omega_Q$ is isomorphic to the given fiber functor of $\D$ by the definition of the simple extension. From here on, we fix $\omega_Q$.
    \begin{prop}\label{prop: QxQ when D is pointed}
        For a braided non-semisimple near-group category $\C$ with a generalized fusion rule $(G,r)$, if $r>0$ then $\E_+ = \D$, i.e. $\D$ is symmetric and $\D' = \C$.
    \end{prop}
    \begin{proof}
        By \cref{prop: values of lambda_g}, we know that $\lambda_g^2 = 1$ for all $g\in G$. Applying \cref{eq: squared braiding on (X times Y) times Z} to the triplet $\{Q,Q,L_g\}$ gives (see also \cref{remark: scalar maps braiding on isomorphic objects}): 
        $$
        s_{Q\otimes Q,L_g} = \lambda_g^2\id_{Q\otimes Q\otimes L_g} = \lambda_g^2\id_{\bar{P}}\oplus (\oplus_r \lambda_g^2\cdot \id_Q)
        $$ 
        On the other hand, using the naturality of the square braiding we have:
        $$
            s_{Q\otimes Q,L_g} = s_{\bar{P},L_g}\oplus (\oplus_r s_{Q,L_g}) = s_{\bar{P},L_g} \oplus (\oplus_r\lambda_g\id_Q)
        $$
    Comparing these two expressions and using our assumption that $r>0$, we conclude that $\lambda_g = 1$ for all $g$. 

    Set $\T:=\Z_2(\D)$. For $X\in \T$, the maps $s_{X,Q}$ define a tensor automorphism $\mu_X$ of the fiber functor $\omega_Q\rvert_\T:\T\to \Vect$ (the monoidality of $\mu_X$ comes from the fact that the objects in $\T$ centralize each other), which in turn corresponds to a grouplike element in the Hopf algebra reconstructed from this fiber functor (see \cite[Ch. 5]{EGNO}). Since $\lambda_g = 1$, $\mu$ is the identity on the simple objects in $\T$, thus $\mu$ is unipotent on every object. Since $\mu$ has finite order, unipotency and $\operatorname{char}(\kk) = 0$ implies $\mu = \id$. This mean that $s_{X,Q}$ acts trivially on $\Hom(Q,X\otimes Q)$, but $X\otimes Q = Q^{\oplus \FPdim(X)}$, so $s_{X,Q} = \id_{X\otimes Q}$, thus $\T\subset \Z_2(\C)$. On the other hand, by \cref{prop: squared braiding on Q}, $Q\notin\Z_2(\C)$ thus $\Z_2(\C)\subset \T$ hence we get an equality. 

    If $Q\notin \D'$ then $\D' = \T = \C'$, and by \cref{eq: shimizu} we get:
    $$
        \FPdim(\D)\FPdim(\T) = \FPdim(\C)\FPdim(\T) \Rightarrow \FPdim(\D) = \FPdim(\C)
    $$
    which is a contradiction, thus $Q\in \D'$, and by \cref{crl: If Q in D' then D symmetric} we are done.
    \end{proof}

    \restate{Theorem}{\ref{thm: non semisimple braided near-group categories have fusion rule 0}}{Let $\C$ be a non-semisimple braided near-group category with generalized fusion rule $(G,r)$, then $r=0$, and consequently $\C$  is weakly integral.}
    \begin{proof}[Proof of \cref{thm: non semisimple braided near-group categories have fusion rule 0}]
        Let $\C$ be a braided non-semisimple near-group category with a generalized fusion rule $(G,r)$, and assume that $r>0$. By \cref{prop: QxQ when D is pointed}, we get that $\D$ is a symmetric subcategory of $\C$ and $Q\in \D'$, so $\D\subset \Z_2(\C)$, which is a contradiction to \cref{prop: D is not central in C}.
        
        By \cref{eq: Q times Q}, if $r=0$ then $Q\otimes Q = \oplus_{g\in G}P_g$ for $P_g$ the projective cover of the simple object $L_g$, and all the projective covers of simple objects have the same Frobenius-Perron dimension, thus $\FPdim(\C) = 2|G|\FPdim(P_e)$ for $P_e$ the projective cover of $\1$, hence $\C$ is weakly integral.
    \end{proof}
    As a result of \cref{thm: non semisimple braided near-group categories have fusion rule 0}, when we consider non-semisimple braided near-group categories $\C$, the generalized fusion rule is always $(\pic(\D),0)$, so as an abbreviation, we will simply say that $\C$ is a $\pic(\D)$-braided near-group category. 
    
    \subsection{De-equivariantization of near-group categories}
    In this subsection, we prove that every $\pic(\D)$-braided near-group category is "an extension" of a POL-category by some symmetric fusion category, that is, there is a braided exact sequence:
    $$
        \Rep(K)\xrightarrow{i}\C\xrightarrow{F}\C_{nd}
    $$
    for $\C_{nd}$ a non-degenerate braided tensor category, such that $\sRep(W)\subset \C_{nd}$ is Lagrangian subcategory for some $W$ purely odd super vector space, with $\dim(W)\ne 0$. This implies that the structure of $\C$ is controlled by two families of well understood braided tensor categories.
    \begin{prop}\label{prop: de-equivariantization of near-group}
        Let $\C$ be a $\pic(\D)$-braided near-group category for some $\D$. Let $K$ be a finite group such that $\Rep(K)\subset \Z_2(\C)$ is a tannakian subcategory, then: 
        \begin{enumerate}
            \item $\C_K = \D_K\oplus \M_K$ as abelian categories.
            \item $\D_K$ is pointed, non semisimple with $\pic(\D_K) = G/\widehat{K}$.
            \item $M_K$ is semisimple with simple objects $Y_1,...,Y_t$, all projective in $\C_K$.
            \item For the forgetful functor $U:\C_K\to \C$ and the free module functor $F:\C\to \C_K$ we have:
            $$
                U(Y_i) = Q^{\oplus e}, \quad F(Q)\cong \oplus Y_i^{\oplus e}
            $$
            and $te^2 = |K|$.
        \end{enumerate}
    \end{prop}
    \begin{proof}
        By \cref{prop: de-equivariantizaion of braided finite tensor categories} $\C_K$ is a well-defined braided finite tensor category, and the free module functor $F:\C\to \C_K$ is a braided tensor functor. This functor is surjective and normal by \cite{BN3}.

        By our assumptions on $K$, for any $X\in \C_K$, as an object $X$ splits to $X_\D\oplus X_\M$, and $A:=\fun(K,\kk)$ is an algebra object in $\D$, so $A\otimes X_\D\in \D$, $A\otimes X_\M\in \M$, thus $X_\D$ and $X_\M$ are $A$-submodules, hence (1) follows.
        
        Let $Y\in \C_K$ be a simple object, and take $L\subset U(Y)$ a simple object. By the adjunction of the forgetful functor and the free module functor we get:
        $$
            0\ne \Hom_{\C}(L,U(Y)) = \Hom_{\C_K}(F(L),Y)
        $$
        If $L\in \D$ then it is invertible, thus by \cref{prop: invertible objects under de-equivariantization} so is $F(L)$, so $Y\cong F(L)$ is invertible and $Y\in \D_K$, thus $\D_K$ is pointed. Since $(\D_K)^K\cong \D$, and an $K$-equivariantization of a fusion category is fusion, we get that $\D_K$ is also non-semisimple. 
        
        If $L = Q$, $F(Q) = Q\otimes A$ is projective in $\C_K$ and $\Hom_{\C_K}(F(Q),Y) = \Hom_{\C}(Q,U(Y))$, thus as an object in $\C$, we have $U(Y) = Z\oplus Q^{\oplus m}$ for some $m\in \N$, $Z\in \D$. By the abelian decomposition of $\C_K$ and the assumption that $Y$ is simple, $Z=0$.

        The algebra $A$ is separable (as an algebra object in $\Rep(K)$), so $Y$ is a direct summand of $Y\otimes A$ in $\Mod_\C(A)$, so to prove that $Y$ is projective in $\Mod_\C(A)$, it is enough to prove that $Y\otimes A$ is projective in $\Mod_\C(A)$, but this is clear because of \cite[Lemma 7.8.12]{EGNO} and the fact that $U(Y)$ is a direct sum of copies of $Q$. 

        For $X\in\D$ set $R(X):=\Hom_{\C}(Q,Q\otimes X)$. Since $\M$ is a right $\D$-module category with a unique simple object $Q$ and $\End(Q)=\kk$, the maps
        $$
            R(X)\otimes R(Y)\to R(X\otimes Y),\qquad f\otimes g\mapsto (f\otimes \mathrm{id}_Y)\circ g,
        $$
        make $R$ a fiber functor on $\D$, and $\Phi:=\Hom_{\C}(Q,-): \M\to \Vect$ is an equivalence with $\Phi(M\otimes X)\cong \Phi(M)\otimes R(X)$ naturally in $M\in\M$, $X\in\D$. Consequently $R(A)$ is an algebra, and $\Phi$ induces an equivalence
        $$
            \mathcal M_K=\Mod_{\M}(A)\simeq\Mod\big(R(A)\big),\qquad M\mapsto \Phi(U(M)),
        $$
        under which the $F(Q)=Q\otimes A$ corresponds to the regular module $R(A)$.

        As an object of $\Rep(K)$, $A=\fun(K,\kk)=\bigoplus_{\chi\in \widehat{K}}\kk\chi$, and the multiplication restricts to isomorphisms $\kk\chi\otimes\kk\psi\to\kk{\chi\psi}$. Each $R(\chi)$ is one-dimensional (as $Q\otimes\chi\cong Q$), so choosing $0\neq u_\chi\in R(\chi)$ we get
        $u_\chi u_\psi=\alpha(\chi,\psi)\,u_{\chi\psi}$ with $\alpha(\chi,\psi)\in\kk^\times$; associativity makes $\alpha$ a $2$-cocycle, i.e. $R(A)\cong \kk_\alpha[\widehat{K}]$ is a twisted group algebra. It is semisimple since $\operatorname{char}\kk=0$, so $R(A)\cong\prod_{i=1}^t\operatorname{Mat}_{e_i}(\kk)$, where $t=\dim Z(R(A))$.

        All $e_i$ are equal: $K$ acts on $R(A)$ by algebra automorphisms $h\cdot u_\chi=\chi(h)u_\chi$.
        If $\varphi,\varphi': Z(R(A))\to\kk$ are two algebra maps, then, since the center is spanned
        by the invertible elements $u_\chi$ lying in it, $\chi\mapsto\varphi'(u_\chi)/\varphi(u_\chi)$ is a character of the subgroup $\{\chi : u_\chi\in Z(R(A))\}\subset\widehat K$; it extends to $\widehat K$, hence is evaluation at some $h\in K$, and $\varphi'=\varphi\circ(h\cdot)$. Thus $K$ permutes the blocks transitively, so $e_i=e$ for all $i$, and $te^2=\dim R(A)=|K|$.

        Let $V_1,\dots,V_t$ be the simple $R(A)$-modules ($\dim V_i=e$) and $Y_i\in\M_K$ the corresponding simple objects. Then $\Phi(U(Y_i))=V_i$, i.e. $U(Y_i)\cong Q^{\oplus e}$, and the decomposition $R(A)\cong\bigoplus_i V_i^{\oplus e}$ of the regular module gives $F(Q)\cong\bigoplus_i Y_i^{\oplus e}$. 

    \end{proof}
    \newpage
    \begin{lemma}\label{lemma: E_+ tannakian}
        Let $\mathcal C$ be a $\pic(\D)$-braided near-group category and $G=\pic(\D)$.
        \begin{enumerate}
            \item $G_Q = \mathbb{Z}_2$ and $[G:\operatorname{Pic}(\E_+)]=2$.
            \item $\mathcal \E_+$ is a Tannakian fusion category, i.e. $\E_+\simeq\Rep(K)$ for a finite abelian group $K$ with $\widehat{K}=\pic(\E_+)$.
        \end{enumerate}
    \end{lemma}
    \begin{proof}
    \begin{enumerate}
        \item By \cref{prop: values of lambda_g}, $\forall g\in G~~\lambda_g^2 = 1$. If $\forall g\in G~~ \lambda_g = 1$, by the proof of \cref{prop: QxQ when D is pointed} $\D = \E_+$, contradicting \cref{prop: D is not central in C}, thus we conclude that there is $g\in G$ for which $\lambda_g\ne 1$, hence $G_Q = \mathbb{Z}_2$. In particular, the claim follows by $G/\pic(\E_+)\cong G_Q$ from the proof of \cref{prop: values of lambda_g}. \item  By \cref{prop: E_+ is central in C}, $\Z_2(C)=E_+\subset\mathcal D$ is a pointed symmetric finite tensor category, so by \cite{D} $\E_+\simeq\sRep(K\ltimes W',u)$, and $K$ is abelian since $\E_+$ is pointed. We show $u=1$, thus $W'=0$ (as $u$ acts on $W'$ by the parity) and $\E_+=\Rep(K)$ is Tannakian.

        Assume $u\neq1$. As $K$ is abelian, there is a character $\chi$ of $K$ with $\chi(u)=-1$; let
        $L_\chi\in\E_+$ be the corresponding invertible object. Then $c_{L_\chi,L_\chi}=-\id$, and $L_\chi^{\otimes 2}$ lies in the Tannakian subcategory $\Rep(\Gamma)\subset\E_+$, $\Gamma:=K/\langle u\rangle$, of objects on which $\langle u\rangle\ltimes W'$ acts trivially. Let $\mathcal L:=\mathcal C_\Gamma$, with free module functor $F$, and put
        $f:=F(L_\chi)$. Since $F$ is braided and every object of $\mathcal L$ is a quotient of some $F(Z)$, $f$ is
        an invertible central object of $\mathcal L$ with $f\otimes f\cong 1$ and $c_{f,f}=-\mathrm{id}$.

        Call the simple summands of $F(Q)$ \emph{odd}; $F(Q)$ is semisimple since $\fun(\Gamma,\kk)$ is separable.
        \begin{enumerate}
            \item By \cref{thm: non semisimple braided near-group categories have fusion rule 0}, $Q\otimes Q\cong\bar P\in\mathcal D$, so $F(Q)\otimes F(Q)\cong F(\bar P)$ has a composition series whose factors are of the form $F(L_g)$, which are invertible by \cref{prop: invertible objects under de-equivariantization}. Let $X,Y$ be odd simple objects. Then $Y\otimes X^*$ is a direct summand of
            $F(Q)\otimes F(Q)^*\cong F(Q\otimes Q)$, so it has a filtration with invertible factors $a_1,\dots,a_k$. Tensoring with $X$ gives a filtration of $Y\otimes X^*\otimes X$ with simple factors $a_i\otimes X$, and the
            evaluation $Y\otimes X^*\otimes X\to Y$ is an epimorphism onto the simple object $Y$. Hence $Y\cong a\otimes X$ for some invertible $a$ (of the form $F(L_g)$).
            \item Let $X$ be odd. Then $X^*$ is odd (as $Q^*\cong Q$), so by (a) $X^*\cong a\otimes X$ for some invertible $a$. For $b:=a^{-1}$,
            $$
                \Hom(b,X\otimes X)\cong\Hom(b\otimes X^*,X)\cong\Hom(X,X)=\kk .
            $$
            Composition with $c_{X,X}$ is an automorphism of $\Hom(b,X\otimes X)$; let $\rho\neq0$ be the corresponding scalar.
            \item The group $\Gamma$ acts on $\mathcal L$ by braided autoequivalences $T_\gamma$
            with $T_\gamma\circ F\cong F$, and $F(U(Z))\cong\bigoplus_{\gamma\in\Gamma}T_\gamma(Z)$ for $Z\in\mathcal L$ (see \cite{BN3,DGNO}).
            Applying this to an odd $Z=X$ with $U(X)\cong Q^{\oplus m}$ shows that every odd simple object is
            isomorphic to some $T_\gamma(X)$. Since $L_\chi\otimes Q\cong Q$, the object $f\otimes X$ is a summand of
            $F(L_\chi\otimes Q)\cong F(Q)$, hence odd; choose $\gamma$ with $T_\gamma(X)\cong f\otimes X$. Also $b\cong F(L_g)$, so
            $T_\gamma(b)\cong b$. As $T_\gamma$ is braided, it carries $\Hom(b,X\otimes X)$ isomorphically onto
            $\Hom(b,fX\otimes fX)$, intertwining $c_{X,X}$ with $c_{fX,fX}$. So $c_{fX,fX}$ acts on         $\Hom(b,fX\otimes fX)$ by $\rho$.
            \item Since $f$ is central, the hexagon axioms show that under the isomorphism $fX\otimes fX\cong f\otimes f\otimes X\otimes X\cong X\otimes X$ (moving $f$ past $X$ and using $f\otimes f\cong \1$)
            the braiding $c_{fX,fX}$ corresponds to $c_{f,f}\otimes c_{X,X}=-c_{X,X}$. Hence $c_{fX,fX}$ acts on $\Hom(b,fX\otimes fX)\cong\Hom(b,X\otimes X)$ by $-\rho$.
        \end{enumerate}
        By (c) and (d), $\rho=-\rho$, so $\rho=0$, a contradiction. Therefore $u=1$.
    \end{enumerate}
\end{proof}
    
    \restate{Theorem}{\ref{thm: exact sequence of braided near-group categories}}{Let $\C$ be a braided non-semisimple near-group category, and let $K$ be the finite abelian group such that $\Z_2(\C)\cong\Rep(K)$. Then
    $$
        \Rep(K)\xrightarrow{\ i\ }\C\xrightarrow{\ F\ }\C_K
    $$
    is a modularization exact sequence, and:
    \begin{enumerate}
        \item $\C_K = \D_K\oplus \M_K$ as abelian categories, $\D_K$ is a pointed non-semisimple braided finite tensor category with $\pic(\D_K)\cong\mathbb Z/2$, and its non-trivial invertible object $S$ satisfies $\Z_2(\D_K)=\langle S\rangle\simeq\sVec$, $\M_K$ is a semisimple category with $t$ simple objects $Y_1,...,Y_t$;
        \item $t\in\{1,2\}$, and if $t=2$ then $Y_2\cong S\otimes Y_1$;
        \item $\D_K$ is braided equivalent to $\sRep(W\oplus W^*)$ with a non-symmetric braiding, for a purely odd super vector space $W$ with $\dim W=n\geq1$, and $\C_K$ contains a Lagrangian subcategory braided equivalent to $\sRep(W)$. In particular, $\C_K$ is a category of purely odd Lagrangian type.
    \end{enumerate}
    $\C_K$ is a braided non-semisimple near-group category if and only if $t=1$.}
    \begin{proof}
        By \cref{thm: non semisimple braided near-group categories have fusion rule 0}, $\C$ is a $\pic(\D)$-braided near-group category. By \cref{prop: the subcategory E+}, \cref{prop: E_+ is central in C} and \cref{lemma: E_+ tannakian}, the category $\E_+$ generated by objects in $\D$ that centralize the unique simple projective object $Q\in \C$ is a tannakian subcategory in $\Z_2(\C)$. By \cref{example: modularization sequence} we get the required braided exact sequence.

        By \cref{prop: de-equivariantization of near-group}(2) and \cref{lemma: E_+ tannakian}, $\pic(\D_K)\cong \pic(\D)/\pic(\E_+)\cong\mathbb Z/2$. $\pic(\E_+)=\widehat{K}$. Denote by $S$ the non-trivial invertible object in $\D_K$ and $\M_K$ is semisimple with simple objects $Y_1,...,Y_t$. The objects $Y_i$ are not invertible: they are projective, and an invertible projective object would make $\1$ projective, hence $\C_K$ semisimple. Thus every invertible object of $\C_K$ lies in $\D_K$, and by the proof of \cref{prop: de-equivariantization of near-group}(1) it is of the form $F(L_x)$, $x\in G$, with $U(F(L_x))\cong\bigoplus_{\chi\in \widehat{K}}L_{x_\chi}$ for $F:\C\to \C_K$ the free module functor and $U:\C_K\to \C$ the forgetful functor. 
        
        Set $A:=\fun(K,\kk)$. By \cref{thm: non semisimple braided near-group categories have fusion rule 0}, $F(Q)\otimes_A F(Q)\cong F(\bar P)$ lies in $\D_K$, and its composition factors are invertible. For $i,j$, the object $Y_j\otimes_A Y_i^*$ is a direct summand of $F(Q)\otimes_AF(Q)^*\cong F(Q\otimes Q)$, so it has a filtration with invertible factors $a_1,\dots,a_k$. Tensoring with $Y_i$ and composing with the evaluation $Y_j\otimes_AY_i^*\otimes_AY_i\to Y_j$, which is an epimorphism onto a simple object, gives $Y_j\cong a_l\otimes_AY_i$ for some $l$. Since $\pic(\D_K)=\{1,S\}$, we get $t\le 2$, and $Y_2\cong S\otimes_A Y_1$ if $t=2$.

        Let $L=F(L_x)$ be invertible with $\tilde s_{Y_i,L}=\id$ for some $i$. Since $U(Y_i)\cong Q^{\oplus e}$, $U(L)\cong\bigoplus_{\chi\in \widehat{K}}L_{x_\chi}$ with $\lambda_{x_\chi} = \lambda_x$ (because $\chi$ is central). Denoting $\pi:U(L)\otimes U(Y_i)\to U(L\otimes_A Y_i)$ we get from \cref{prop: de-equivariantizaion of braided finite tensor categories}:
        $$
        \pi = \tilde{s}_{L,Y_i}\circ \pi = \pi\circ s_{U(L),U(Y_i)} = \lambda_x\pi
        $$ 
        where the first equality is by the choice of $L$. From this we deduce that $\lambda_x = 1$, meaning $x\in \widehat{K} = \ker(\lambda)$ and $L\cong F(L_x)\cong \1$. This implies that $S$ doesn't centralize any $Y_i$, so $Y_i$ is not central for any $i$, hence the only simple object in $\Z_2(\C_K)$ is $\1_{\C_K}$. Since $\operatorname{char}(\kk)=0$ we have $\operatorname{Ext}^1_{\C_K}(\1,\1)=0$ by \cite[Theorem 2.17]{EO}, thus $\Z_2(\C_K) = \Vect$, i.e. $\C_K$ is non-degenerate.
        
        By \cref{eq: Q times Q} and \cref{eq: shimizu}, we have $2\FPdim(\D) = \FPdim(\C)$. Since $F$ divides the Frobenious-Perron dimensions of the categories by $|K|$, we have $2\FPdim(\D_K)=\FPdim(\C_K)$, so by applying \cref{eq: shimizu} for $\D_K$ in $\C_K$ we get
        $$
            \FPdim(\D_K)\FPdim((\D_K)_{\C_K}')=\FPdim(\C_K)\FPdim(\Z_2(\C_K)\cap\D_K)=\FPdim(\C_K),
        $$
        thus $\FPdim((\D_K)_{\C_K}')=2$, so $\D_0:=(\D_K)_{\C_K}'$ has $2$ invertible simple objects, thus $\D_0\cong \sVec$ or $\D_0\cong \Rep(\mathbb{Z}_2)$. Since $\D_0\subset \D_K$, we conclude that $\D_0\subset \Z_2(\D_K)$. $\D_K$ is non-semisimple and also have $2$ simple objects, thus if $\D_0\cong \Rep(\mathbb{Z}_2)$, it is a tannakian central subcategory of $\D_K$, which implies that a $\mathbb{Z}_2$-de-equivariantization of $\D_K$ exists, and is a non-semisimple tensor category with a unique simple object, which is impossible in $\operatorname{char}(\kk)=0$, thus $\D_0\cong \sVec$.    
        
        $\D_K$ is integral, thus $\D_K\cong \Rep(H)$ for some finite dimensional quasi-Hopf algebra $H$, whose radical has codimension $2$ (see for instance \cite{EGNO}). Its semisimple part $\langle S\rangle\simeq\sVec$ is symmetric and equivalent (as a tensor category without braiding) to $\Rep(\mathbb{Z}_2)$, so the associator of $\langle S\rangle$ is trivial; by \cite[Theorem 3.1]{EG2} this quasi-Hopf algebra is twist equivalent to a Nichols Hopf algebra, i.e.\ $\D_K$ admits a fiber functor.\footnote{If $t=1$ this is immediate: $\M_K=\langle Y_1\rangle$ is a $\D_K$-module category of rank one.} By the classification of pointed braided finite tensor categories with a fiber functor, $\D_K$ is determined by a non-zero purely odd vector space $V$ with an alternating bilinear form, and the condition $\Z_2(\D_K)=\sVec$ means that the form is non-degenerate \cite{BN1}. Hence $\dim V=2n$ with $n\geq1$, and $\D_K\simeq\sRep(W\oplus W^*)$ with its non-symmetric braiding, where $W\subset V$ is a Lagrangian subspace. The subspace $W$ gives a symmetric tensor subcategory $\mathcal L\simeq\sRep(W)$ of $\D_K$ \cite{BN1} with $\FPdim(\mathcal L)=2^{n+1}$, while $\FPdim(\D_K)=2^{2n+1}$ and $\FPdim(\C_K)=2^{2n+2}=\FPdim(\mathcal L)^2$. By \cref{eq: shimizu} and non-degeneracy, $\FPdim(\mathcal L')=\FPdim(\C_K)/\FPdim(\mathcal L)=\FPdim(\mathcal L)$; as $\mathcal L\subset\mathcal L'$, we get $\mathcal L=\mathcal L'$ \cite{EGNO}, i.e. $\mathcal L$ is Lagrangian. In particular, $\C_K$ is non-degenerate braided tensor category with a Lagrangian subcategory braid equivalent to $\sRep(W)$, thus a POL-category.

        The last statement follows from (2): $\C_K=\D_K\oplus\M_K$ with $\M_K$ semisimple, and $\M_K$ has a unique simple object if and only if $t=1$.
    \end{proof}
        
    The following corollary is immediate from the proof of \cref{thm: exact sequence of braided near-group categories}, for no de-equivariantization is needed.
    \begin{crl}\label{crl: non-degenerate near-group category is POL}
        If $\C$ is a non-degenerate braided non-semisimple near-group category, it is of purely odd Lagrangian type.
    \end{crl}
    \subsection{The structure of a braided non-semisimple near-group category} Using \cref{thm: exact sequence of braided near-group categories}, we can better describe the algebraic data defining a braided non-semisimple near-group category $\C$.

    \begin{lemma}\label{lemma: H is a 2-group}
        In the setting of \cref{thm: exact sequence of braided near-group categories}, $e=2^m$ for some $m\geq0$, and $K$ is a $2$-group with $|K|=t\,2^{2m}$.        
    \end{lemma}
    \begin{proof}
        By the proof of \cref{prop: de-equivariantization of near-group} and the adjunction of $F$ and $U$, the algebra $E:=\End_{\C_K}(F(Q))$ is (anti-)isomorphic to $\kk_\alpha[\widehat K]$ as an $\widehat K$-graded algebra, where the grading is the eigenspace decomposition for the action of $K$ on $\mathcal C_K$ (note $T_h\circ F\cong F$ as tensor functors). Write $u_xu_y=b(x,y)u_yu_x$ for homogeneous basis elements; $b$ is an alternating bicharacter, and $u_x$ is central iff $x\in R:=\operatorname{rad}(b)$, so $t=|R|$. 
        
        Fix $\varphi:Q^*\xrightarrow{\sim}Q$. Then $\tau(f):=F(\varphi)\circ f^*\circ F(\varphi)^{-1}$ is an anti-automorphism of $E$, and it commutes with the action of $K$, since $T_h\circ F\cong F$ is monoidal and hence compatible with duals. So $\tau$ preserves the grading, and applying $\tau$ to $u_xu_y=b(x,y)u_yu_x$ gives $b(y,x)=b(x,y)$. As $b$ is alternating, $b(x,y)^2=1$ for all $x,y$. Hence $b(x^2,y)=1$, i.e. $x^2\in R$, so $\widehat K/R$ is an elementary abelian $2$-group carrying the non-degenerate alternating form induced by $b$. Thus $|\widehat K/R|=2^{2m}$, $e^2=|\widehat K|/t=2^{2m}$, and $|K|=t\,2^{2m}$ with $t\le 2$ by \cref{thm: exact sequence of braided near-group categories}.    \end{proof}

    \begin{prop}\label{prop: invariants of braided near-group categories}
        Let $\C$ be a braided non-semisimple near-group category, $G=\pic(\C)$, $p=\FPdim(P(\1))$ (the projective cover of $\1_\C)$ and $q=\FPdim(Q)$. Then there are integers $n\geq1$ and $k\geq1$ such that
        $$
            p=2^{2n},\qquad |G|=2^{k},\qquad q=\sqrt{|G|\,p}=2^{\,n+k/2},\qquad \operatorname{FPdim}(\mathcal C)=2^{2n+k+1}.
        $$
        In particular, $G$ is an abelian $2$-group, and $\C$ is integral if and only if $k$ is even.
    \end{prop}

    \begin{proof}
        By \cref{eq: Q times Q}, $q^2=\FPdim(\oplus_{g\in G}P(L_g))=|G|\cdot p$ and $\FPdim(\C)=2|G|\cdot p$.
        Let $K$, $t$, $e$ be as in \cref{thm: exact sequence of braided near-group categories} and \cref{prop: de-equivariantization of near-group}. By \cref{lemma: E_+ tannakian} $|G|=2|K|=2te^2$, which is a power of $2$ by \cref{lemma: H is a 2-group}; write $|G|=2^k$, and note that $k$ is odd if $t=1$ and even if $t=2$. By \cref{thm: exact sequence of braided near-group categories}(3), $\FPdim(\D_K)=2^{2n+1}$ with $n=\dim W\geq 1$, and $\FPdim(\D)=|K|\FPdim(\D_K)$, so
        $p=\FPdim(\D)/|G|=2^{2n+1}|K|/(2|K|)=2^{2n}$. Since all simple objects other than $Q$ are invertible, $\C$ is integral if and only if $q=2^{n+k/2}$ is an integer, i.e. if and only if $k$ is even.
    \end{proof}

    \begin{thm}
        Let $\C$ be a braided non-semisimple near-group category. Then $\C$ determines a pair $(n_g,n_p)$ of positive integers such that:
        \begin{enumerate}
            \item $G=\pic(\C)$ is an abelian $2$-group of order $2^{n_g}$;
            \item $\dim_{\kk}(P(\1))=2^{n_p}$, and $n_p$ is even;
            \item $n_p+n_g$ is odd if and only if $\C$ is not integral.
        \end{enumerate}
    \end{thm}
    \begin{proof}
    \cref{prop: invariants of braided near-group categories} with $n_g=k$, $n_p=2n$, together with $\dim_{\kk}(X)=\FPdim(X)$ for $X\in \D$ since $\D$ is integral.
    \end{proof}

    \begin{crl}
        $\Rep(H_4)$ has no braided simple extension.
    \end{crl}
    \begin{proof}
        In $\Rep(H_4)$ the projective cover of $\1$ has dimension $2=2^1$, but $n_p$ must be even for such an extension to exist.
    \end{proof}


\begin{thebibliography}{9}
    \bibitem{EGNO} P. Etingof, S. Gelaki, D. Nikshych, V. Ostrik, Tensor categories, \textit{Mathematical Surveys and Monographs}, vol. \textbf{205} (2015).

    \bibitem{S1} J. Siehler, Near-group categories, \textit{Algebr. Geom. Topol.} \textbf{3} (2003), 719--775.

    \bibitem{TY} D. Tambara, S. Yamagami, Tensor categories with fusion rules of self-duality for finite abelian groups, \textit{J. Algebra} \textbf{209} (1998), 692--707.

    \bibitem{S2} J. Siehler, Braided near-group categories, \textit{arXiv:math/0011037} (2000).

    \bibitem{T} J. Thornton, On braided near-group categories, \textit{arXiv:1102.4640} (2011).

    \bibitem{M} W. Michihisa, Various structures associated to the representation categories of eight-dimensional non-semisimple Hopf algebras, \textit{Algebras and Representation Theory} \textbf{7}(4) (2004), 491--515.

    \bibitem{BN1} C. G. Boneta, D. Nikshych, Pointed braided tensor categories, \textit{Tensor categories and Hopf algebras, Contemp. Math.}, Amer. Math. Soc., Providence, RI, {\bf 728} (2019), 67--94.

    \bibitem{GS} S. Gelaki, D. Sebbag, On finite non-degenerate braided tensor categories with a Lagrangian subcategory, \textit{Trans. Amer. Math. Soc. Ser. B} \textbf{9} (2022), 450--469.

    \bibitem{BN2} A. Bruguières, S. Natale, Exact sequences of tensor categories, \textit{Int. Math. Res. Not.} \textbf{24} (2011), 5644--5705.

    \bibitem{BN3} A. Bruguières, S. Natale, Central exact sequences of tensor categories, equivariantization and applications, \textit{J. Math. Soc. Japan} \textbf{66}(1) (2014), 257--287.

    \bibitem{DGNO} V. Drinfeld, S. Gelaki, D. Nikshych, V. Ostrik, On braided fusion categories I, \textit{Selecta Math.} \textbf{16}(1) (2010), 1--119.

    \bibitem{Sh} K. Shimizu, Non-degeneracy conditions for braided finite tensor categories, \textit{Adv. Math.} \textbf{355} (2019), 106778.

    \bibitem{SY} K. Shimizu, H. Yadav, Commutative exact algebras and modular tensor categories, \textit{Selecta Mathematica} \textbf{32} no. 78 (2026), doi 10.1007/s00029-026-01187-1.
    
    \bibitem{EO} P. Etingof, V. Ostrik, Finite tensor categories, \textit{Moscow Math. J.} \textbf{4} (2004), 627--654.

    \bibitem{EG} P. Etingof, S. Gelaki, Exact sequences of tensor categories with respect to a module category, \textit{Advances in Mathematics} \textbf{308} (2017), 1187--1208.

    \bibitem{D} P. Deligne, Catégories tensorielles, \textit{Moscow Math. J.} {\bf 2}  (2002) no. 2, 227--248.

    \bibitem{GHS} A. M. Gainutdinov, J. Haferkamp, C. Schweigert, Davydov--Yetter cohomology, comonads and Ocneanu rigidity, \textit{Advances in Mathematics} \textbf{414} (2023).

    \bibitem{EG2} P. Etingof, S. Gelaki, Finite dimensional quasi-Hopf algebras with radical of codimension 2, \textit{Mathematical Research Letters} \textbf{11} (2004), 685--696.

\end{thebibliography}
\end{document}